\documentclass{amsart}
\usepackage{amsthm}
\usepackage{amsmath}
\usepackage{amstext}
\usepackage{amsfonts}
\usepackage{latexsym}
\usepackage{amscd}
\usepackage{xypic}
\theoremstyle{plain}
\newtheorem{theorem}{Theorem}[section]
\newtheorem{lemma}[theorem]{Lemma}
\newtheorem{proposition}[theorem]{Proposition}
\newtheorem{corollary}[theorem]{Corollary}

\theoremstyle{definition}
\newtheorem{definition}[theorem]{Definition}

\newtheorem{problem}[theorem]{Problem}
\newtheorem{example}[theorem]{Example}
\newtheorem{remark}[theorem]{Remark}
\newtheorem{remarks}[theorem]{Remarks}

\numberwithin{equation}{theorem}
\newtheorem{claim}[theorem]{Claim}
\begin{document}

\title[Automorphisms of canonically polarized surfaces in char 2]{Automorphisms of smooth canonically polarized surfaces in characteristic 2.}
\author{Nikolaos Tziolas}
\address{Department of Mathematics, University of Cyprus, P.O. Box 20537, Nicosia, 1678, Cyprus}
\email{tziolas@ucy.ac.cy}
%\thanks{This paper was partially written during the author's stay at the Max Planck Institute for Mathematics in Bonn, from April to August 2009.}

%    General info
%\date{April 21, 1999}

\dedicatory{This paper is dedicated to my wife Afroditi and my son Marko.}


\begin{abstract}
This paper investigates the structure of the automorphism scheme of a smooth canonically polarized surface $X$ defined over an algebraically closed field of characteristic 2. In particular it is investigated when $\mathrm{Aut}(X)$ is not smooth. This is a situation that appears only in positive characteristic and it is closely related to the structure of the moduli stack of canonically polarized surfaces. Restrictions on certain numerical invariants of $X$ are obtained in order for $\mathrm{Aut}(X)$ to be smooth or not and information is provided about the structure of the component of $\mathrm{Aut}(X)$ containing the identity. In particular, it is shown that if $X$ is a smooth canonically polarized surface with $1\leq K_X^2\leq 2$ with non smooth automorphism scheme, then $X$ is uniruled. Moreover, if $K_X^2=1$, then $X$ is simply connected, unirational and  $\mathrm{p}_g(X) \leq 1$. Moreover, $X$ is the purely inseparable quotient of a rational surface by a rational vector field.
\end{abstract}

\maketitle

\section{Introduction}
One of the most important problems in algebraic geometry is the classification up to isomorphism of algebraic varieties defined over an algebraically closed field $k$. In order to deal with this problem, the class of varieties is divided into smaller classes where a reasonable answer is expected to exist. Quite generally, once such a class $\mathcal{C}$ is chosen, the corresponding moduli functor $\mathcal{M}_{\mathcal{C}}\colon Sch(k) \rightarrow (Sets)$ is defined by setting for any $k$-scheme $S$, $\mathcal{M}_{\mathcal{C}}(S)$ to be the set of isomorphism classes of flat morphisms $X\rightarrow S$ whose fibers are in $\mathcal{C}$. Sometimes, depending on the case, some extra structure on the family morphisms are required to create a ``reasonable'' moduli functor. The best that one could hope for is that the functor is representable, which means hat 
there is a universal family $\mathcal{X} \rightarrow \mathrm{M}_{\mathcal{C}}$ such that any other family is obtained from it by base change. In this case $\mathrm{M}_{\mathcal{C}}$ is called a fine moduli space for $\mathcal{M}_{\mathcal{C}}$. Unfortunately fine moduli spaces rarely exist. The reason for this failure is the presence of nontrivial automorphisms of the objects that one wants to parametrize.  

One way of dealing with this difficulty is instead of looking for a universal family, to settle for less and search for a variety $\mathrm{M}_{\mathcal{C}}$ whose $k$-points are in one to one correspondence with the varieties of the moduli problem and which satisfies some uniqueness property. Such a variety is called a coarse 
moduli space. However, the biggest disadvantage of this approach is that usually the coarse moduli space does not support a family and therefore it gives very little 
information about families of varieties in the moduli problem. In order to study families as well, the universality condition is relaxed and one looks for a so called modular family. Loosely speaking a modular family is a family $X \rightarrow S$ such that up to \'etale base change, any other family is obtained from it by base change and that for any closed point $s \in S$, the completion $\hat{\mathcal{O}}_{S,s}$ prorepresents the local deformation functor $Def(X_s)$. In some sense a modular family is a connection between the local moduli functor (which behaves well) and the global one. In modern language one says that the moduli stack associated to the moduli problem is Deligne-Mumford~\cite{DM69}.
  
In dimension 1 curves are separated by their genus $g$. The moduli functor $\mathcal{M}_g$ of smooth curves of genus $g$ defined over an algebraically closed field has a coarse moduli space $\mathrm{M}_g$. For $g=0$ it is a reduced point, for $g=1$ it is the $j$-line and for $g\geq 2$ it is irreducible of dimension $3g-3$ and it admits a compactification $\bar{\mathrm{M}}_g$ whose boundary points correspond to stable curves, i.e, to reduced curves $C$ with at worst nodes as singularities and $\omega_C$ ample~\cite{DM69}. In all cases the corresponding moduli stack is Deligne-Mumford. 
   
In dimension 2, surfaces are divided according to their kodaira dimension $\kappa$ which takes the values $-\infty, 0, 1$ and $ 2$. Surfaces with kodaira dimension 2 are the corresponding cases to the case of curves of genus $\geq 2$ and are called surfaces of general type. Early on in the theory of moduli of surfaces of general type in characteristic zero, it was realized that the correct objects to parametrize are not the surfaces of general type themselves but their canonical models.  For compactification reasons, the moduli functor is extended to include the so called stable surfaces. These are reduced two dimensional schemes $X$ with semi-log-canonical (slc) singularities such that there is an integer $N$ such that $\omega^{[N]}_X$ is ample. They are the higher dimensional analog of stable curves. Then for any fixed integer valued function $H(m)$  the moduli functor $\mathcal{M}^s_{H} \colon Sch(k) \rightarrow (Sets)$ of stable surfaces with fixed Hilbert polynomial is defined~\cite{KSB88}~\cite{Ko10}. 
It is known 
that $\mathcal{M}^s_{H}$ has a separated coarse moduli space $\mathrm{M}^s_{H}$ which is of finite type over the base field $k$. Moreover, the corresponding moduli stack is a separated, proper Deligne-Mumford stack of finite type~\cite{KSB88}~\cite{Ko97}~\cite{Ko10}.  

This paper is supposed to be a small contribution into the study of the moduli of canonically polarized surfaces in positive characteristic. It also inspires to bring attention to the positive characteristic case and motivate people to work on it. In particular, it is a first attempt to study the following general problem.

\begin{problem}
Study the moduli stack of stable varieties with fixed Hilbert polynomial defined over an algebraically closed field of characteristic $p>0$. In particular, is it proper, Deligne-Mumford or of finite type? In the case when one of these properties fails, why does it fail and how can the moduli problem be modified in order for the corresponding stack to satisfy them?
\end{problem}

As mentioned earlier, there has been tremendous progress in the characteristic zero case lately. In particular, the case of surfaces has been completely settled. However, the positive characteristic case is to the best of my knowledge a wide open area. The main reason is probably the many complications and pathologies that appear in positive characteristic. For example, Kodaira vanishing fails and the minimal model program and semistable reduction, two ingredients essential in the compactification of the moduli in the characteristic zero case, are not known, at the time of this writing, to work in positive characteristic. Moreover, the moduli stack of smooth canonically polarized surfaces is not Deligne-Mumford. The reason for this failure in positive characteristic is the existense of smooth canonically polarized surfaces with non smooth automorphism scheme~\cite{La83},~\cite{SB96},~\cite{Li08}.  This does not happen in characteristic zero simply because every group scheme in characteristic zero is smooth. 
Therefore the non smoothness of the automorphism scheme is an essential obstruction for the moduli stack of canonically polarized surfaces to be Deligne-Mumford. 

However, in the case of surfaces there is reason to believe that it would be possible to obtain a good moduli theory even in positive characteristic. Resolution of singularities exists for surfaces. Moreover, canonical models of surfaces are classically known to exist and the semistable minimal model program for semistable threefolds holds~\cite{Kaw94} . In addition, the definition of stable surfaces is characteristic free and therefore the definition of the functor of stable surfaces applies in positive characteristic too. Finally, canonically polarized surfaces with fixed Hilbert polynomial are bounded~\cite{Ko84} and Koll\'ar's quotient results show that there exist a separated coarse moduli space of finite type for the moduli functor of canonically polarized surfaces over $\mathrm{Spec}\mathbb{Z}$ with a fixed Hilbert polynomial~\cite{Ko84}. 

Based on these observations, the first step in the study of the moduli of canonically polarized surfaces in positive characteristic is to investigate why and when the automorphism scheme is not smooth. This investigation will tell us how the moduli functor can be modified in order to obtain proper Deligne-Mumford stacks and how to deal with surfaces with non smooth automorphism schemes. Naive restrictions of the moduli problem to surfaces with smooth automorphism scheme does not work because this property is not deformation invariant, as shown by example~\ref{ex2}. In any case, it is essential to study surfaces with non smooth automorphism scheme, simply because they exist.
  
The purpose of this paper is to investigate the structure or the automorphism scheme of a smooth canonically polarized surface defined over an algebraically closed field of characteristic 2, and in particular to find conditions under which the automorphism scheme is either smooth or not. In particular to find deformation invariant conditions for the smoothness of the automorphism scheme. This way one could construct a proper Deligne-Mumford substack of the moduli stack of stable surfaces. The reasons that this study is restricted to characteristic 2 are purely of technical nature and will be explained later.

The main result of this paper is the following.
\begin{theorem}\label{main-theorem}
Let $X$ be a smooth canonically polarized surface defined over an algebraically closed field of characteristic 2. Suppose that $\mathrm{Aut}(X)$ is not smooth. Then one of the following happens.
\begin{enumerate}
\item $K_X^2 \geq 3$.
\item $K_X^2=2$ and $X$ is uniruled. Moreover, if $\chi(\mathcal{O}_X)\geq 2$, then $X$ is unirational and $\pi_1^{et}(X)=\{1\}$.
\item $K_X^2=1$, $\pi_1^{et}(X)=\{1\}$, $p_g(X) \leq 1$ and $X$ is unirational. 
\end{enumerate}
Moreover, if $1 \leq K_X^2\leq 2$, then $X$ is the quotient of a ruled or rational surface (maybe singular) by a rational vector field.

Finally, if $2\leq K_X^2 \leq 4$ and $\mathrm{Aut}(X)$ contains $\mu_2$ as a subgroup scheme, then $X$ is uniruled and $\pi_1^{et}(X)=\{1\}$. 
\end{theorem}
If $\mathrm{Aut}(X)$ is not smooth then it contains either $\mu_2$ or $\alpha_2$. Theorem~\ref{main-theorem} shows that $\mathrm{Aut}(X)$ having a subgroup isomorphic to $\mu_2$ is a much more restrictive condition as having one isomorphic to $\alpha_2$.

\begin{corollary}\label{main-corollary}
Let $X$ be a smooth canonically polarized surface defined over an algebraically closed field of characteristic 2. Suppose that either $K_X^2=2$ and $X$ is not uniruled, or $K_X^2=1$ and one of the following happens
\begin{enumerate}
\item $p_g(X)=2$.
\item $\chi(\mathcal{O}_X) =3$.
\item $\pi_1^{et}(X) \not= \{1\}$.
\item $X$ is not unirational.
\end{enumerate}
Then $\mathrm{Aut}(X)$ is smooth.
\end{corollary}

In particular, Theorem~\ref{main-theorem} applies to the case of Godeaux surfaces. Considering their significance I find it appropriate to have a statement for this case.
\begin{corollary}\label{godeaux}
Let $X$ be a canonically polarized Godeaux surface defined over an algebraically closed field of characteristic 2. Let $\mathrm{Pic}^{\tau}(X)$ be the torsion sub group scheme of $\mathrm{Pic}(X)$. Suppose that $\mathrm{Aut}(X)$ is not smooth. Then $X$ is unirational, $\pi_1^{et}(X)=\{1\}$ and $\mathrm{Pic}^{\tau}(X)$ is either reduced of order $2^n$, $n \in \mathbb{N}$, or isomorphic to $\alpha_2\times N$, where $N$ is a finite reduced commutative group scheme of order $2^n$, $n\in \mathbb{N}$.
\end{corollary}
The group scheme structure of $\mathrm{Pic}^{\tau}(X)$ in the case that $\mathrm{Aut}(X)$ is not smooth follows from the fact that $\pi_1^{et}(X)=\{1\}$,  and the discussion about the possible structure of it in section~\ref{sec-0}. From the characteristics zero and five cases that have been extensively studied~\cite{Re78},~\cite{La83},\cite{Li09}, one might expect that if $\mathrm{Pic}^{\tau}(X)$ is reduced, then it is either $\mathbb{Z}/2\mathbb{Z}$ or $\mathbb{Z}/4\mathbb{Z}$. If on the other hand it is not reduced, then $\mathrm{Pic}^{\tau}(X)\cong \alpha_2$. However, there is no classification in the characteristic 2 case yet and since many pathologies appear in characteristic two, one must be careful. I believe it will be interesting to know if there are examples of canonically polarized surfaces with nonreduced automorphism scheme but reduced Picard scheme.

Finally, the conditions of Corollary~\ref{main-corollary} on the euler characteristic and the \'etale fundamental group are deformation invariant and hence we get the following.
\begin{corollary}
Let $\mathcal{M}_{1,3}$ and $\mathcal{M}_{1,ns}$ be the moduli stacks of canonically polarized surfaces $X$ with $K_X^2=1$ and either $\chi(\mathcal{O}_X)=3$ or $\pi_1^{et}(X)\not= \{1\}$, respectively. Then $\mathcal{M}_{1,3}$ and $\mathcal{M}_{1,ns}$ are Deligne-Mumford stacks.
\end{corollary}

Next I would like to discuss how restrictive and how effective the conditions of Theorem~\ref{main-theorem} are. First of all, there are examples of surfaces of general type with arbitrary large $K^2$ and non reduced automorphism scheme~\cite{La83},~\cite{Li09}. In fact in characteristic 2 there are non uniruled examples as well~\cite{SB96}. Moreover, Theorem~\ref{th1} suggests that nonsmoothness of the automorphism scheme happens for small values of $K^2$ compared to the characteristic of the base field. In particular, in characteristic 2, the class of surfaces with $K^2=1$ must be the most pathological class and is of particular interest to the problem. Moreover, one cannot expect that it will be possible to get results similar to those of Theorem~\ref{main-theorem} for any surface with nonreduced automorphism scheme  because there are examples of canonically polarized surfaces that are not uniruled or simply connected and yet they have nonreduced automorphism scheme~\cite{SB96}. However, due to the lack 
of examples, I do not know if 
surfaces with $K^2\leq 2$ is the maximal class of surfaces that Theorem~\ref{main-theorem} holds. 

Suppose that $X$ is a canonically polarized surface $X$ with $K_X^2\leq 2$. If $K_X^2=2$, then by~\cite{Ek87}, it follows that $0 \leq \chi(\mathcal{O}_X)\leq 4$. Moreover, if $\chi(\mathcal{O}_X) \geq 3$, then by Lemma~\ref{b1}, $\pi_1^{et}(X)=\{1\}$ and hence the condition that $X$ is simply connected in Theorem~\ref{main-theorem}.2 has value only for $\chi(\mathcal{O}_X) =2$.

Suppose that $K_X^2=1$. Then it is well known~\cite{Li09} that  $\mathrm{p}_g(X)\leq 2$, $1 \leq \chi(\mathcal{O}_X) \leq 3$ and $|\pi_1^{et}(X)| \leq 6$. In characteristic zero there are examples for all cases and similar examples are expected to exist in characteristic 2. Hence many surfaces with $K_X^2=1$ are excluded in the theorem and hence by the Corollary~\ref{main-corollary} have smooth automorphism scheme. For 
example, Godeaux surfaces that are quotients of a smooth quintic in $\mathbb{P}^3_k$ by a free action of $\mathbb{Z}/5\mathbb{Z}$, since they are not simply connected. 

At this point I would like to mention that it is hard to find examples of smooth canonically polarized surfaces with low $K^2$ and non smooth automorphism scheme. In fact, I believe that the biggest disadvatage of this paper is the lack of examples of surfaces with norreduced automorphism scheme and low $K^2$ which will show how effective the results of Theorem~\ref{main-theorem} are. N. I. Shepherd-Barron~\cite{SB96} has constructed an example of a smooth canonically polarized surface $X$ with non smooth automorphism scheme and $K_X^2=8$. This is the example with the lowest $K^2$ that I know. Simply connected Godeaux surfaces exist in all characteristics~\cite{LN12} but it is not known if their automorphism scheme is smooth or not. Singular examples are much easier to find. Two such examples are presented in section~\ref{examples}

Based on the previous discussion, I believe it would be an interesting problem to search for examples of canonically polarized surfaces with non smooth automorphism scheme and low $K^2$. In particular, Corollary~\ref{godeaux} motivates the following problem.
\begin{problem}
Are there any simply connected Godeaux surfaces with non smooth automorphism scheme? Equivalently, are there any with non trivial global vector fields? If there are then what is the torsion part of their picard scheme?
\end{problem}
A complete classification of Godeaux surfaces in characteristic 2 would be a very interesting problem too.

The main steps in the proof of Theorem~\ref{main-theorem} are the following. 

In Section~\ref{sec-2} it is shown that $\mathrm{Aut}(X)$ is smooth if and only if $X$ admits a nonzero global vector field $D$ such that either $D^2=0$ or $D^2=D$. Equivalently if $X$ admits a nontrivial $\alpha_2$ or $\mu_2$ action. Moreover, it is shown that if $X$ lifts to characteristic zero this does not happen.

In Section~\ref{sec-3} quotients of a smooth surface $X$ by an $\alpha_2$ or $\mu_2$ action are studied. In particular the singularities of the quotient $Y$ are described. Proposition~\ref{prop4} shows that there is an essential diffference between $\alpha_2$ quotients and $\mu_2$ quotients. In the $\mu_2$ case $Y$ has only canonical singularities of type $A_1$. However, in the $\alpha_2$ case, $Y$ may have even non rational singularities. This makes it necessary to consider the two cases separately.

In Section~\ref{sec-4} it is investigated which smooth canonically polarized surfaces $X$ admit vector fields of multiplicative type. The results are presented in Theorem~\ref{mult-type}. The method used in order to study this problem is the following. Suppose that $X$ has a global vector field $D$ of multiplicative type. Then $D$ induces a $\mu_2$ action on $X$. Let $\pi \colon X \rightarrow Y$ be the quotient and $g \colon Y^{\prime} \rightarrow Y$ its minimal resolution. The results are obtained by considering cases with respect to the Kodaira dimension $k(Y)$ of $Y$ and comparing invariants and the geometry of $X$ and $Y$. The basic idea of this method was first used by Rudakov and Shafarevich~\cite{R-S76} in order to show that a smooth K3 surface has no global vector fields.

In Section~\ref{sec-5} it is investigated which smooth canonically polarized surfaces $X$ admit vector fields of additive type. The results are presented in Theorem~\ref{additive-type}. The general lines of the proof of the theorem are as in the proof of Theorem~\ref{mult-type}. However, there are many complications arising from the extra, compared to the multiplicative case, possible singularities of $Y$ which require different arguments, especially in the cases $k(Y)=0$ and $k(Y)=-\infty$.

Theorem~\ref{main-theorem} is the combination of the results of Theorems~\ref{mult-type} and~\ref{additive-type}. At this point I would like to point out that it would be possible to derive Theorem~\ref{main-theorem} only by using the methods of Theorem~\ref{additive-type}. However, Theorems~\ref{mult-type},~\ref{additive-type} provide detailed information about when a smooth canonically polarized surface $X$ admits a global vector field of either multiplicative or additive type. Moreover, if $\mathrm{Aut}(X)$ is not smooth then it would be interesting to know its group scheme structure, in particular that of its connected component containing the identity. It is a hard problem to determine its structure completely. However, finding its subgroups is easier and gives a lot of information about it. If $\mathrm{Aut}(X)$ is not reduced, then it contains either $\mu_2$ or $\alpha_2$, or both. Theorem~\ref{mult-type} gives conditions under which $\mathrm{Aut}(X)$ does not contain $\mu_2$ and Theorem~\ref{additive-
type} gives conditions under which $\mathrm{Aut}(X)$ does not contain $\alpha_2$.
 
In Section~\ref{examples} an overview of known examples is given. Moreover two examples are given of singular surfaces $X$ and $Y$ with non smooth automorphism scheme and  $K_X^2=1$, $K_Y^2=5$. $X$ has singularities of index 2 and $Y$ has canonical singularities of type $A_n$. The significance of these examples are twofold. First singular surfaces should be studied because they are important in the compactification of the moduli problem. The first example shows that if there are no restrictions on the singularities then $K^2$ can be anything. The second has canonical singularities and is therefore a stable surface, a surface that is in the moduli problem, with low $K^2$. Moreover, $Y$ is smoothable to a smooth canonically polarized surface with smooth automorphism scheme. This shows that the property ``smooth automorphism scheme'' is not deformation invariant and does not produce a proper moduli stack.

Finally I would like to explain the reasons that this study has been restricted to characteristic 2. The main reason is that in characteristic 2, it is possible to control the singularities of the quotient $Y$ of a smooth surface $X$ by a $\mu_2$ or $\alpha_2$ action. The singularities of the quotient $Y$ correspond to the isolated singularities of the vector field $D$ inducing the action. By~\cite{Hi99} the singularities can be resolved by a series of blow ups. This is not possible in higher characteristics. Moreover, in characteristic 2 $X$ is a torsor over a codimension 2 open subset of $Y$. This is not always true in higher characteristics. On the other hand characteristic 2 has its own unique difficulties. For example many pathologies, like the existence of quasi-elliptic fibrations,  appear only in characteristics 2 and 3. Moreover, nonclassical Godeaux surfaces, that play an important role in Theorem~\ref{main-theorem} exist only in characteristics $p \leq 5$~\cite{Li09}. Finally, in the view of 
Theorem~\ref{th1}, the structure of the automorphism scheme should be worse for small characteristics. From this point of view, 2 is the worst case.

\section{Preliminaries.}\label{sec-0}
Let $X$ be a scheme of finite type over a field $k$ of characteristic $p>0$.

$X$ is called a smooth canonically polarized surface if and only if $X$ is a smooth surface and $\omega_X$ is ample.

$Der_k(X)$ denotes the space of global $k$-derivations of $X$ (or equivalently of global vector fields). It is canonically identified with $\mathrm{Hom}_X(\Omega_X,\mathcal{O}_X)$.

A nonzero global vector field  $D$ on $X$ is called of additive or multiplicative type if and only if $D^p=0$ or $D^p=D$, respectively. A prime divisor $Z$ of $X$ is called an integral divisor of $D$ if and only if locally there is a derivation $D^{\prime}$ of $X$ such that $D=fD^{\prime}$, $f \in K(X)$,  $D^{\prime}(I_Z)\subset I_Z$ and $D^{\prime}(\mathcal{O}_X) \not\subset I_Z$ ~\cite{R-S76}. 

Let $\mathcal{F}$ be a coherent sheaf on $X$. By $\mathcal{F}^{[n]}$ we denote the double dual $(\mathcal{F}^{\otimes n})^{\ast\ast}$. 

For any prime number $l\not= p$, the cohomology groups $H_{et}^i(X,\mathbb{Q}_l)$ are independent of $l$, they are finite dimensional of $\mathbb{Q}_l$ and are called the $l$-adic cohomology groups of $X$. The $i$-Betti number $b_i(X)$ of $X$ is defined to be the dimension of $H_{et}^i(X,\mathbb{Q}_l)$. It is well known that $b_i(X)=0$ for any $i>2n$, where $n=\dim X$~\cite{Mi80}. 

The \'etale Euler characteristic of $X$ is defined by \[
\chi_{et}(X)=\sum_{i}(-1)^i\dim_{\mathbb{Q}_l}H^i(X,\mathbb{Q}_l)=\sum_i(-1)^ib_i(X).
\]
If $X$ is a smooth surface then $c_2(X)=\chi_{et}(X)$~\cite{Mi80}. Both the Betti numbers and the \'etale euler characteristic are invariant under \'etale equivalence. In particular, if $f \colon X \rightarrow Y$ is a purely inseparable morphism of varieties, then $f$ induces an equivalence of the \'etale sites of $X$ and $Y$ and hence $b_i(X)=b_i(Y)$ and $\chi_{et}(X)=\chi_{et}(Y)$~\cite{Mi80}.

$X$ is called algebraically simply connected if $\pi_1^{et}(X)=\{1\}$.

We will use the terminology of terminal, canonical, log terminal and log canonical singularities as in~\cite{KM98}. Their definition and basic properties, in particular~\cite[Corollary4.2, Corollary 4.3, Theorem 4.5]{KM98} are independent of the characteristic of the base field and therefore their theory applies in positive characteristic too. The contraction theorems~\cite{Art62},~\cite{Art66} are also independent of the characteristic and will be used frequently in this paper.

Let $P\in X$ be a normal surface singularity and $f \colon Y \rightarrow X$ its minimal resolution. If $P\in X$ is canonical, then $K_Y=f^{\ast}K_X$. By~\cite{KM98} canonical surface singularities are classified according to the Dynkin diagrams of their minimal resolution and they are called accordingly of type $A_n$, $D_n$, $E_6$, $E_7$ and $E_8$. In characteristic zero these are exactly the DuVal singularities and their dynkin diagrams correspond to explicit equations. However in characteristic 2 I am not aware of a classification with respect to local equations. In this paper they will be distinguished according to their dynkin diagrams   

A Godeaux surface is a surface of general type with the lowest possible numerical invariants and it is a classical object of study (at least in characteristic zero). More precisely,
\begin{definition}\cite{Re78}
A numerical Godeaux surface is a minimal surface $X$ of general type such that $K_X^2=1$ and $\chi(\mathcal{O}_X)=1$. 
\end{definition}
Numerical Godeaux surfaces are divided in two disjoint classes. \textit{Classical} and \textit{Nonclassical}. A Godeaux surface $X$ is called classical if and only if the torsion part $\mathrm{Pic}^{\tau}(X)$ of $\mathrm{Pic}(X)$ is reduced, and nonclassical if and only if $\mathrm{Pic}^{\tau}(X)$ is not reduced. Nonclassical are further divided into two disjoint classes. \textit{Singular} if $\mathrm{Pic}^{\tau}(X)$ contains $\mu_p$ and \textit{Supersingular} if $\mathrm{Pic}^{\tau}(X)$ contains $\alpha_p$.

Since all group schemes are reduced in characteristic zero, nonclassical Godeaux surfaces appear only in positive characteristic $p>0$. It is well known that in characteristic zero  $\mathrm{Pic}^{\tau}(X)$ can be $\{1\}$, $\mathbb{Z}/2\mathbb{Z}$,  $\mathbb{Z}/3\mathbb{Z}$,$\mathbb{Z}/4\mathbb{Z}$ and  $\mathbb{Z}/5\mathbb{Z}$~\cite{Re78}. In positive characteristic however, $\mathrm{Pic}^{\tau}(X)$ may have a non reduced part and one gets either singular or supersingular Godeaux surfaces too. Classical Godeaux surfaces have been classified in characteristic zero by Miles Reid~\cite{Re78}, and in characteristic $p=5$ by W. Lang~\cite{La83}. Nonclassical in characteristic 5 have been classified by C. Liedtke~\cite{Li09}. There is no classification yet in characteristic 2 but it is expected that all the characteristic zero cases appear also in characteristic 2 together with cases where $\mathrm{Pic}^{\tau}(X)=\alpha_2$ or $\mu_2$. 

There is a correspondence between the structure of $\mathrm{Pic}^{\tau}(X)$ and the existence of torsors over $X$.  The correspondence is given by the isomorphism $\mathrm{Hom}(G, \mathrm{Pic}^{\tau}(X))\cong H^1_{fl}(X,G^{\ast})$~\cite{Ray70}, where $G$ is any commutative subgroup scheme of $\mathrm{Pic}^{\tau}(X)$ and $G^{\ast}$ is its Cartier dual. According to this, a Godeaux surface is singular if there is a $\mathbb{Z}/2\mathbb{Z}$ torsor and supersingular if there is an $\alpha_2$ torsor over $X$. Equivalently, if the induced map of the Frobehious on $H^1(\mathcal{O}_X)$ is either bijective or zero.

Finally, $\mathbb{F}_p$ denotes the finite field with $p$ elements.
\section{Nonreducedness happens for small p.}\label{sec-1}

The purpose of this section is to show that nonreducedness of the automorphism scheme of a canonically polarized normal surface defined over an algebraically closed field of characteristic $p>0$ is a property that happens for relatively small values of $p$. Moreover, the length of the automorphism schemes of canonically polarized surfaces with a fixed Hilbert polynomial is bounded by a number that depends only on the Hilbert polynomial and not on the characteristic of the base field. In particular I will show the following.

\begin{theorem}\label{th1}
Let $f(x)\in \mathbb{Q}[x]$ be a numerical polynomial. Then there are positive integers $m_0$ and $M_0$ depending only on $f(x)$ such that for any Gorenstein canonically polarized surface $X$ defined over an algebraically closed field $k$ of characteristic $p>0$ and with Hilbert polynomial $f(x)$,
\begin{enumerate} 
\item \[
\mathrm{length}(\mathrm{Aut}(X)) \leq M
\]
\item $\mathrm{Aut}(X)$ is smooth for all $p>m_0$. 
\end{enumerate}
\end{theorem}

\begin{proof}
Let $\Omega$ be the set of all Gorenstein canonically polarized surfaces with fixed Hilbert polynomial $f(n)$ defined over any field of any characteristic. Then this set is bounded~\cite{Ko84}~\cite{M70}~\cite{M-M64}. This means that there is a flat morphism $f \colon \mathcal{X} \rightarrow S$, where $S$ is of finite type over $\mathbb{Z}$ whose geometric fibers are Gorenstein canonically polarized surfaces and such that for any Gorenstein canonically polarized surface $X$ with Hilbert polynomial $f(n)$ defined over an algebraically closed field $k$, there is a morphism $\mathrm{Spec}k \rightarrow S$ such that $X \cong \mathrm{Spec}k \times_S \mathcal{X}$. Let
\[
\Phi \colon \mathrm{Aut}(\mathcal{X}/S) \rightarrow S
\]
be the induced morphism on relative automorphism schemes. I will show that this map is finite. For this it suffices to show that $\Phi$ is proper with finite fibers. It is well known that for any canonically polarized surface $X$, $\mathrm{Aut}(X)$ is a finite group scheme. Therefore $\Phi$ has finite fibers. Properness of $\Phi$ follows from the valuative criterion of properness. This is equivalent to the following property. Let $R$ be a discrete valuation ring with function field $K$ and residue field $k$ (perhaps of mixed characteristic). Let $X \rightarrow \mathrm{Spec}R$ be a projective flat morphism such that $\omega_{X/R}$ is ample. Let $X_K$ and $X_k$ be the generic and special fibers. Then any automorphism of $X_K$ lifts to an automorphism of $X$ over $\mathrm{Spec}R$. The proof of this statement is identical with the one for characteristic zero~\cite[Proposition 3]{Ko10} and I omit its proof. It essentially depends on the existence of resolutions of singularities which exist also in 
any characteristic in dimension two.

Now since $\Phi$ is a finite morphism the lengths of its fibers are bounded by some number $M$ that depends only on $f(x)$. Hence if $X$ is any Gorenstein canonically polarized surface defined over an algebraically closed field $k$ of characteristic $p>0$, $\mathrm{length}(\mathrm{Aut}(X))\leq M$. Therefore if $p>M$, $\mathrm{Aut}(X)$ is smooth over $k$. This follows from the fact that $\mathrm{Aut}(X)$ is a finite group scheme defined over a field of characteristic $p>0$ and any such group scheme is smooth if $p$ is bigger than its length~\cite{Mu70}~\cite{Mi12}. This concludes the proof of the theorem.
\end{proof}
The previous theorem motivates the following problem.

\begin{problem}
Find effective bounds for the length of the automorphism scheme, or its component containing the identity, of a canonically polarized surface, depending only on its Hilbert polynomial.
\end{problem}

\begin{remarks}
\begin{enumerate}
\item The previous theorem shows that the structure of the automorphism scheme is expected to be more complicated in small characteristics. It also suggests that the characteristic 2 case may be the case where most pathologies appear.
\item The proof of the theorem depends on boundedness of canonically polarized surfaces with a given Hilbert polynomial. At the time of this writing this is not known in higher dimensions.  
\end{enumerate}

\end{remarks}

\section{Nonreducedness of the automorphism scheme, derivations and group scheme actions.}\label{sec-2}
Let $X$ be a canonically polarized $\mathbb{Q}$-Gorenstein surface defined over an algebraically closed field $k$ of characteristic $p>0$. The purpose of this section is essentially  to show that $\mathrm{Aut}(X)$ is not smooth over $k$ if and only if $X$ admits a nontrivial global vector field or equivalently if it admits a nontrivial $\mu_p$ or $\alpha_p$ action. These results are easy and probably known but I include them here for lack of reference and the convenience of the reader.

\begin{proposition}\label{prop1}
Let $X$ be a canonically polarized $\mathbb{Q}$-Gorenstein surface defined over an algebraically closed field $k$ of characteristic $p>0$. Then $\mathrm{Aut}(X)$ is not smooth over $k$ if and only if $X$ admits a nontrivial global vector field of either additive or multiplicative type.
\end{proposition}

\begin{proof}
It is well known that for any canonically polarized $\mathbb{Q}$-Gorenstein surface, $\mathrm{Aut}(X)$ is finite. Therefore it is not smooth if and only if its tangent space at the identity is not trivial. Now the tangent space of $\mathrm{Aut}(X)$ at the identity is $\mathrm{Hom}_X(\Omega_X,\mathcal{O}_X)$ which is the space of global derivations of $X$. Therefore $\mathrm{Aut}(X)$ is not smooth if and only if $X$ has a nontrivial global vector field. 

Now let \[
\Phi \colon \mathrm{Hom}_X(\Omega_X,\mathcal{O}_X) \rightarrow \mathrm{Hom}_X(\Omega_X,\mathcal{O}_X)
\]
be the map defined by $\Phi(D)=D^p$. This is a $p$-linear map. Therefore by~\cite[Lemma 4.13]{Mi80}, there is a decomposition $\mathrm{Hom}_X(\Omega_X,\mathcal{O}_X)=V_n\oplus V_s$ where $V_n$, $V_s$ are $\Phi$-invariant subspaces such that the restriction of $\Phi$ to $V_n$ is nilpotent and to $V_s$ bijective. If $V_n \not= \emptyset$ then $X$ admits a vector field of additive type and if $V_s\not= \emptyset$ then it admits a vector field of multiplicative type. 
\end{proof}
\begin{corollary}\label{lifts-to-zero}
Let $X$ be a smooth canonically polarized surface which lifts to characteristic zero, or to $W_2(k)$, the ring of 2-Witt vectors. Then $\mathrm{Aut}(X)$ is smooth.
\end{corollary}
\begin{proof}
Since $X$ lifts to characteristic zero or $W_2(k)$, then Kodaira-Nakano vanishing holds for $X$~\cite{DI87},~\cite{EV92} and hence
\[
H^2(T_X)=H^0(\Omega_X\otimes \omega_X^{-1})=0,
\]
since $\omega_X$ is ample. Hence $X$ has no nontrivial global vector fields and therefore by Proposition~\ref{prop1}, $\mathrm{Aut}(X)$ is smooth.
\end{proof}
This result is a kind of accident since the smoothness of the automorphism scheme is not the consequence of any vanishing theorems but rather the fact that any group scheme is smooth in characteristic zero. Moreover, this result does not hold for singular surfaces as shown by the examples in section 6. 
\begin{proposition}\label{prop2}
Let $X$ be a canonically polarized $\mathbb{Q}$-Gorenstein surface defined over an algebraically closed field $k$ of characteristic $p>0$. Then $\mathrm{Aut}(X)$ is not smooth over $k$ if and only if $X$ admits a nontrivial $\alpha_p$ or $\mu_p$ action.
\end{proposition}
\begin{proof}
By Proposition~\ref{prop1}, $\mathrm{Aut}(X)$ is not smooth if and only if $X$ has a nontrivial vector field $D$ of either additive or multiplicative type. Define the map
\[
\Phi \colon \mathcal{O}_X \rightarrow \frac{\mathcal{O}_X[t]}{(t^p)}
\]
by setting 
\[
\Phi(a)=\sum_{k=0}^{p-1}\frac{D^ka}{k!}
\]
If $D$ is of additive type (i.e., $D^p=0$) then $\Phi$ defines an action $\alpha_p$ on $X$ and if it is of multiplicative type (i.e., $D^p=D$) then it defines an action of $\mu_p$ on $X$.

Suppose now that $X$ admits a nontrivial $\alpha_p$ or $\mu_p$ action. I will show that $X$ admits a nontrivial vector field $D$ of either additive or multiplicative type, respectively. I will only do the case when an $\alpha_p$ action exists. The other is similar and is ommited. 

Suppose that $X$ admits a nontrivial $\alpha_p$-action. Let \[
\mu \colon \alpha_P \times X \rightarrow X
\]
be the map that defines the action. Let 
\[
\mu^{\ast} \colon \mathcal{O}_X \rightarrow \mathcal{O}_X \otimes_k \frac{k[t]}{(t^p)}
\]
be the corresponding map on the sheaf of rings level. The additive group $\alpha_p$ is $\mathrm{Spec}\frac{k[t]}{(t^p)}$ as a scheme with group scheme structure given by
\[
m^{\ast}\colon \frac{k[t]}{(t^p)} \rightarrow \frac{k[t]}{(t^p)} \otimes \frac{k[t]}{(t^p)} \]
defined by $m^{\ast}(t)=1\otimes t + t \otimes 1$. Then by the definition of group scheme action, there is a commutative diagram.
\[
\xymatrix{
\mathcal{O}_X \ar[r]^{\mu^{\ast}} \ar[d]^{\mu^{\ast}} & \mathcal{O}_X \otimes_k \frac{k[t]}{(t^p)} \ar[d]^{\mu^{\ast}\otimes 1} \\
\mathcal{O}_X \otimes_k \frac{k[t]}{(t^p)} \ar[r]^{1\otimes m^{\ast}} & \mathcal{O}_X \otimes_k \frac{k[t]}{(t^p)} \otimes_k \frac{k[t]}{(t^p)}\\
}
\]
The map $\mu^{\ast}$ is given by \[
\mu^{\ast}(a)= a\otimes 1 +\Phi_1(a)\otimes t +\sum_{k=2}^{p-1}\Phi_k(a)\otimes t^k,
\] 
where $\Phi_k \colon \mathcal{O}_X \rightarrow \mathcal{O}_X$ are additive maps. The fact that $\mu^{\ast}$ is a sheaf of rings map shows that $\Phi_1$ is a derivation which we call $D$. I will show by induction that $\Phi_k(a)=D^ka/k!$. From this it follows that any $\alpha_p$ action is induced by a global vector field $D$ as in the first part of the proof.

From the commutativity of the previous diagram and the definition of $\mu^{\ast}$ and $m^{\ast}$ it follows that
\begin{gather*}
a \otimes 1 \otimes 1 + Da \otimes t \otimes 1 +\sum_{k=2}^{p-1} \Phi_k(a) \otimes t^k \otimes 1+ Da \otimes 1 \otimes t +D^2a\otimes t \otimes t + \\
\sum_{k=2}^{p-1}\Phi_k(Da)\otimes t^k 
\otimes t + \sum_{k=2}^{p-1}\left( \Phi_k(a)\otimes 1 \otimes t^k +D(\Phi_k(a))\otimes t \otimes t^k +\sum_{s=2}^{p-1}\Phi_s(\Phi_k(a))\otimes t^s \otimes t^k\right) =\\
a\otimes 1 \otimes 1 +Da \otimes t \otimes 1 + Da \otimes 1 \otimes t +\sum_{k=2}^{p-1}\sum_{s=0}^{k}\Phi_k(a) \binom{k}{s}t^s \otimes t^{k-s}
\end{gather*}
Equating the coefficients of $t^{k-1}\otimes t $ on both sides of the equation we get that $k\Phi_k(a)t^{k-1}\otimes t = \Phi_{k-1}(Da)\otimes t^{k-1} \otimes t$. By induction, 
$\Phi_{k-1}=D^{k-1}/(k-1)!$. Therefore $\Phi_k(a)=D^ka/k!$ as claimed.

\end{proof}
From the previous proposition immediately follows that,
\begin{corollary}\label{subgroup-of-aut}
Let $X$ be a canonically polarized $\mathbb{Q}$-Gorenstein surface defined over an algebraically closed field $k$ of characteristic $p>0$. Then $\mathrm{Aut}(X)$ is not smooth if and only if it has a subgroup scheme isomorphic to either $\alpha_p$ or $\mu_p$.
\end{corollary}

\section{Quotients by $\alpha_p$ or $\mu_p$ actions.}\label{sec-3}
In this section $X$ denotes a scheme of finite type over an algebraically closed field $k$ of characteristic $p>0$ which admits a nontrivial $\alpha_p$ or $\mu_p$ action. As explained in the previous section, such an action is induced  by a global nontrivial vector field $D$ of $X$. Let $\pi \colon X \rightarrow Y$ be the quotient, which exists as an algebraic scheme by~\cite{Mu70}. The purpose of this section is to describe the structure of the map $\pi$ and the singularities of $Y$.

\begin{definition}\cite{Sch07}
\begin{enumerate}
\item The fixed locus of the action of $\alpha_p$ or $\mu_p$ (or of $D$) on $X$ is the closed subscheme of $X$ defined by the ideal sheaf generated by $D(\mathcal{O}_X)$.
\item A point $P \in X$ is called an isolated singularity of $D$ if there is an embeded component $Z$ of the fixed locus of $D$ such that $P \in Z$. The vector field $D$ is said to have only divisorial singularities if the ideal $D(\mathcal{O}_X)$ has no embeded components.
\end{enumerate}
\end{definition}

The next proposition gives some information about the singularities of $Y$. 
\begin{proposition}\label{prop3}
Let $X$ be an integral scheme of finite type over an algebraically closed field of characteristic $p>0$. Suppose $X$ has an $\alpha_p$ or $\mu_p$ action induced by a vector field $D$ of either additive or multiplicative type. Let $\pi \colon X \rightarrow Y$ be the quotient. Then
\begin{enumerate}
\item If $X$ is normal then $Y$ is normal.
\item If $X$ is $S_2$ then $Y$ is $S_2$ as well.
\item If $X$ is smooth then the singularities of $Y$ are exactly the image of the embedded part of the fixed locus of the action.
\item If $X$ is normal and $\mathbb{Q}$-Gorenstein, then $Y$ is also $\mathbb{Q}$-Gorenstein. In particular, let $D$ be a divisor in $Y$ and $\tilde{D}$ be the divisorial part of $\pi^{-1}(D)$. Then if $n\tilde{D}$ is Cartier, $pnD$ is Cartier too.
\end{enumerate}
\end{proposition}

\begin{proof}
Normality of $Y$ is a local property so we may assume that $X$ and $Y$ are affine. Let $X=\mathrm{Spec}A$ and $Y=\mathrm{Spec}B$, where $B=\{a\in A, \; Da=0\} \subset A$. Let $\bar{B} \subset K(B)$ be the integral closure of $B$ in its function field $K(B)$. Let $z=b_1/b_2 \in \bar{B}$. Then since $K(B)\subset K(A)$ and $A$ is normal, $z\in A$. But $Dz=D(b_1/b_2)=(b_2Db_1-b_1Db_2)/b_2^2=0$. Therefore $z\in B$ and hence $B$ is integrally closed.

The fact that if $X$ is $S_2$, so is $Y$ is proved in ~\cite{Sch07} and the statement that if $X$ is smooth then the singularities of $Y$ are exactly the image of the embeded part of the fixed locus of $D$ is in~\cite{AA86}.

Suppose that $X$ is $\mathbb{Q}$-Gorenstein. Let $D$ be a divisor on $Y$. The property that $D$ is $\mathbb{Q}$-Cartier is local so we may assume that $X$ and $Y$ are affine, say $X=\mathrm{Spec}A$ and $Y=\mathrm{Spec}B$. Then $D$ is $\mathbb{Q}$-Cartier if and only if $nD=0$ in $\mathrm{Cl}(B)$, for some $n \in \mathbb{N}$. Consider the natural map \[
\phi\colon \mathrm{Cl} (B) \rightarrow \mathrm{Cl} (A)
\]
Then according to~\cite{Fo73}, \[
\mathrm{Ker}\phi \subset  H^1(G, A^{\ast}+K(A)t)
\]
where $G$ is the additive subgroup of $Der_k(A)$ generated by $D$, $A^{\ast}+K(A)t \subset K(A)[t]/(t^2)$ is the multiplicative subgroup and $G$ acts on it by the usual automorphisms induced by $D$. Then since $k$ has characteristic $p>0$, $G \cong \mathbb{Z}/p\mathbb{Z}$  and therefore $H^1(G, A^{\ast}+K(A)t)$ is $p$-torsion. Therefore $\mathrm{Ker}\phi$ is $p$-torsion as well. Hence if $n\tilde{D}$ is $\mathbb{Q}$-Cartier, then $nD\in \mathrm{Ker}\phi$ and hence $pnD=0$ in $\mathrm{Cl} (B)$. Therefore $pnD$ is Cartier as claimed. 
\end{proof}

Even if $X$ is smooth, the singularities of $Y$ are very hard to describe in general. However, in the case of a $\mu_p$ action or if $p=2$, there is the following result.
\begin{proposition}\cite{Hi99}\label{prop4}
Suppose $X$ is a smooth surface with a nontrivial global vector field $D$ or either additive or multiplicative type. Then
\begin{enumerate}
\item If $D$ is of multiplicative type then the singularities of $Y$ are toric singularities of type $\frac{1}{p}(1,m)$, $m=1,2,\ldots, p-1$.
\item If $p=2$ then $Y$ is Gorenstein. Moreover, if $Y$ has canonical singularities, then $Y$ has singularities of type either $A_1$, $D_{2n}$, $E_7$ or $E_8$.
\item If $p=2$ then the isolated singularities of $D$ can be resolved by successive blow ups of the isolated singular points. In particular, there is a diagram
\[
\xymatrix{
X^{\prime}\ar[r]^f \ar[d]^{\pi^{\prime}} & X \ar[d]^{\pi} \\
Y^{\prime} \ar[r]^g & Y \\
}
\]
where $f$ is successive blow ups of the isolated singularities of $D$, $D$ lifts to a derivation $D^{\prime}$ in $X^{\prime}$ with only divisorial singularities, and $Y^{\prime}$ is the quotient of $X^{\prime}$ by $D^{\prime}$( which is smooth since $D^{\prime}$ has only divisorial singularities). 
Moreover,
\begin{enumerate}
\item Every $f$-exceptional curve is contained in the divisorial part of $D^{\prime}$.
\item If $D$ is of multiplicative type, then $Y^{\prime}$ is the minimal resolution of $Y$. Moreover, the divisorial part of $D$ is smooth, disjoint from the isolated singular points of $D$ and is not an integral curve of $D$.
\end{enumerate}
\end{enumerate}
\end{proposition}

\begin{proof}
Everything in the statement of the proposition except 2, 3.(a) and 3.(b) was proved by M. Hirokado~\cite{Hi99}.

Suppose that $p=2$. Then $\pi$ factors through the geometric Frobenious $F \colon X \rightarrow X^{(2)}$. In fact there is a commutative diagram
\[
\xymatrix{
     & Y \ar[dr]^{\nu} \\
X \ar[ur]^{\pi}\ar[rr]^F & & X^{(2)}
}
\]
Since $X^{(2)}$ is smooth and $Y$ is normal, then $\nu$ is a torsor over $X^{(2)}$~\cite{Ek87}. In particular, $Y$ has hypersurface singularities and therefore it is Gorenstein.

Suppose that $Y$ has canonical singularities. Then the dynking diagram of any singular point of $Y$ is of type eithet $A_n$, $D_n$, $E_6$, $E_7$ or $E_8$~\cite{KM98}. By Proposition~\ref{prop3}.4, the local Picard groups of the singular points of $Y$ are 2-torsion. Therefore these can be only $A_1$, $D_{2n+1}$, $E_7$ or $E_8$. 

Suppose that $D$ is of multiplicative type. Let $\Delta$ be its divisorial part. Then in suitable local coordinates of a point $P \in X$, $D$ is given by $D=ax\partial /\partial x +by \partial /\partial y$, where $a, b \in \mathbb{F}_p$~\cite{R-S76}. This shows immediately that the divisorial part $\Delta$ of $D$ is smooth, it is disconnected from the isolated singular points of $D$ and is not an integral curve of $D$. Therefore, if $\Delta^{\prime}$ is the image of $\Delta$ in $Y$ with reduced structure, then $\pi^{\ast}\Delta^{\prime}=2\Delta$. Moreover, it is a straightforward calculation to find the lifting $D^{\prime}$ of $D$ on the blow up $X^{\prime}$ of $X$ at $P$ and see that indeed the exceptional curve is contained in the divisorial part of the fixed locus of $D^{\prime}$. 

Next I will show that $Y^{\prime}$ is the minimal resolution of $Y$. The $g$-exceptional curves are exactly the images under $\pi^{\prime}$ of the $f$-exceptional curves. Let $E$ be an $f$-exceptional currve and $F$ its image. Then since $E$ is not an integral curve for $D^{\prime}$, $(\pi^{\prime})^{\ast}F=2E$ and therefore $4E^2=2F^2$. Therefore, if $F^2=-1$, then $E^2=-1/2$, which is impossible. Hence $Y^{\prime}$ is the minimal resolution of $Y$.

It remains to show that if $D$ is of additive type then every $f$-exceptional curve is contained in the divisorial part of $D^{\prime}$. The map $f$ is obtained by successively blowing up the isolated singular points of $D$. In order then to show that every $f$-exceptional curve is contained in the divisorial part of the fixed locus of $D^{\prime}$, it suffices to assume that $f$ is a single blow up and show that $D^{\prime}$ induces the zero derivation on the exceptional curve $E$. 

Suppose $P$ is on the divisorial part of the fixed locus. Then in suitable local coordinates, $D=h(f \partial /\partial x +f \partial / \partial y)$ such that $f,g,h \in m_P$, where $m_P$ is the maximal ideal of $\mathcal{O}_{X,P}$ and $f,g$ have no common factor. Then $Dx =fh \in m_P^2$ and $Dy=hg\in m_P^2$. Let $E$ be the $f$-exceptional curve. Then $E= \mathrm{Proj} R$, where 
\[
R= \oplus_{d\geq 0} m_P^d/m_p^{d+1}
\]
and $D$ induces a graded derivation of $R$. But since $D(m_P) \subset m_P^2$, it follows that in fact the induced derivation is the zero derivation.

Suppose now that $P$ is an isolated singular point of $D$ that does not belong on the divisorial part of the fixed locus. Then again in suitable local coordinates, $D=f \partial /\partial x +g \partial / \partial y$ such that $f,g\in m_P$ have no common factor. I will show that $f, g \in m_P^2$, and hence $D(\mathcal{O}_X)\subset m_p^2$. Indeed, since we are in characteristic 2, $\partial^2/\partial x^2=\partial^2/\partial y^2 =0$. Then an easy calculation shows that
\[
D^2=\left( f\frac{\partial f}{\partial x}+g\frac{\partial f}{\partial y}\right) \frac{\partial}{\partial x} +\left( f\frac{\partial g}{\partial x}+g\frac{\partial g}{\partial y}\right) \frac{\partial}{\partial y}.
\]
Now the relation $D^2=0$ implies that 
\begin{eqnarray*}
f\frac{\partial f}{\partial x}=g\frac{\partial f}{\partial y} & \text{and} & f\frac{\partial g}{\partial x}=g\frac{\partial g}{\partial y}.
\end{eqnarray*}

Suppose that at least one of $f,g$ is not in $m_P^2$. Suppose $f \in m_P-m_P^2$. Then $\frac{\partial f}{\partial x}\not= 0$ and $\frac{\partial f}{\partial y}\not= 0$ because if $\frac{\partial f}{\partial x}=\frac{\partial f}{\partial y}= 0$, then $f(x,y)=f_1(x^2,y^2)$ and hence $f\in m_P^2$. Now considering that $f$ and $g$ have no common factor, it follows that there is $\phi \in \mathcal{O}_X$ such that 
\begin{eqnarray*} 
\frac{\partial f}{\partial x} = g \phi & \text{and} & \frac{\partial f}{\partial y} = f \phi. \\
\end{eqnarray*}
Now since $f \in m_P-m_P^2$, it follows that
\[
f(x,y)=ax+by +f_{\geq 2}(x,y),
\]
where $f_{\geq 2}(x,y)\in m_P^2$, $a,b, \in k$, not both zero. Then either $\partial f/\partial x \in \mathcal{O}_X^{\ast}$ or $\partial f/\partial y \in \mathcal{O}_X^{\ast}$. But then either $g \in \mathcal{O}_X^{\ast}$ or $f \in \mathcal{O}_X^{\ast}$, which is impossible. Hence $D(m_P)\subset m_P^2$. Now argueing exactly as in the case when the singular point $P$ is on the divisorial part of the fixed locus of $D$ we get that the lifting $D^{\prime}$ of $D$ on $X^{\prime}$ restricts to zero on the exceptional curve.

\end{proof}

\begin{remarks}
If $D$ is of additive type it might happen that, unlike in the multiplicative case, the divisorial part of it is an integral divisor of $D$. For example, let $D=x^2\partial /\partial x +xy \partial /\partial y$. It is easy to see that in characteristic 2, $D^2=0$, its divisorial part is given by $x=0$ and it is an integral curve of $D$. 

The possibility that the divisorial part of $D$ is an integral divisor of $D$ suggests that in the additive case, $Y^{\prime}$ may not be the minimal resolution of $Y$. However, this does not happen if $D$ is of multiplicative type and this is one of the reasons why $\mu_p$ quotients are easier to study.

The previous proposition suggests that, despite the pathologies that the characteristic 2 case has, it has the advantage over other characteristics that the singularities of the quotients can be studied with the help of the commutative diagram in the second statement of the previous proposition.
\end{remarks}

The next proposition describes the structure of the map $\pi$ and shows that over a codimension 2 open subset of $Y$ it is the normalization of an $\alpha_p$ or $\mu_p$ torsor over $Y$.

\begin{proposition}\cite{Tz14}\label{structure-of-pi}
Let $X$ be a normal Cohen Macauley integral scheme of finite type over an algebraically closed field $k$ of characteristic $p>0$ with a $\mu_p$ or $\alpha_p$ action. Let $\pi \colon X \rightarrow Y$ be the quotient. Then there exists a factorization \[
\xymatrix{
X \ar[dr]^{\pi} \ar[rr]^{\phi} & & Z\ar[dl]^{\delta} \\
& Y & \\
}
\]
such that;
\begin{enumerate}
\item Suppose that $X$ admits a $\mu_p$ action. Then there exists a rank 1 reflexive sheaf $L$ on $Y$ and a section $s \in H^0(L^{[p]})$ such that 
$Z=\mathrm{Spec}_Y \left( \oplus_{i=1}^{p-1} L^{[-i]}\right)$ is the $p$-cyclic cover defined by $L$ and $s$ and $X$ is the normalization of $Z$.  Moreover,
\[
\omega_Z=\left(\delta^{\ast}(\omega_Y \otimes L^{p-1})\right)^{[1]}.
\]
\item Suppose that $X$ admits a $\alpha_p$ action. Then there exists a short exact sequence of reflexive sheaves on $Y$
\[
0 \rightarrow \mathcal{O}_Y\stackrel{i}{\rightarrow} E \rightarrow L \rightarrow 0
\]
where $E$ and $L$ have ranks 2 and 1 respectively, a $p$-linear map $\sigma \colon E \rightarrow \mathcal{O}_Y$ and a section $s \in \mathrm{Hom}_Y(L,\mathcal{O}_Y)$, such that $Z=\mathrm{Spec}_YR(E,\sigma)$, where $R(E,\sigma)=S(E)/I$, and $I \subset S(E)$ is the ideal generated by $a-i(a)$, for all $a \in \mathcal{O}_Y$, and $\mathrm{Ker}\Phi$, where $\Phi \colon S(E) \rightarrow \mathcal{O}_Y$ is the map induced on $S(E)$ by $\sigma$. $X$ is the normalization of $Z$ and 
\[
\omega_Z=\left(\delta^{\ast}(\omega_Y \otimes L^{1-p})\right)^{[1]}.
\]
\end{enumerate}
If $p=2$ then in all of the above cases, $X=Z$ even if $X$ is not normal. In particular, $X$ is a torsor in codimension 2. 
\end{proposition}
The case of interest in this paper is the case when $p=2$. In this case the quotient map $\pi \colon X \rightarrow Y$ is a torsor in codimension 2 and the adjunction formulas stated in the theorem hold for $X$. The fact that in the characteristic 2 case if $X$ is normal then $\pi$ is a torsor over a codimension 2 subset of $Y$ was shown by T. Ekedhal~\cite{Ek86}.   
However the previous theorem applies also to non normal varieties and moreover it gives more specific information about $\alpha_p$ and $\mu_p$ quotients.

The next lemma and proposition relate the size of the singular locus of $Y$ with certain numerical invariants of $X$ in the case when $X$ is a smooth surface.

\begin{lemma}\label{ex-seq-1}
Let $D$ be a global vector field on a smooth surface $X$ defined over an algebraically closed field $k$ of characteristic $p>0$. Let $I_Z $ be the ideal sheaf of the embeded part $Z$ or the fixed locus of $D$ and let $\Delta$ be its divisorial part. Then 
\begin{enumerate}
\item There exists an exact sequence
\[
0 \rightarrow \mathcal{O}_X(\Delta) \rightarrow T_X \rightarrow L \otimes I_Z \rightarrow 0,
\]
where $L$ is an invertible sheaf on $X$.
\item Let $ P \in X$ be an isolated fixed point of $D$. Then locally in the \'etale topology  $\mathcal{O}_X=k[x,y]$,  $D=h(f\partial/\partial x +g \partial/\partial y)$, where $h,f,g \in k[x,y]$ and $f,g$ are relatively prime. Moreover, \[
\mathcal{O}_Z=\frac{\mathcal{O}_X}{(f,g)}.
\] 
In particular, if $D$ is of multiplicative type then $Z$ is reduced and its length is equal to the number of isolated fixed points of $D$.
\end{enumerate}
\end{lemma}
\begin{proof}
The vector field $D$ induces a short exact sequence
\[
0 \rightarrow \mathcal{O}_X(\Delta) \rightarrow T_X \rightarrow \mathcal{F} \rightarrow 0, \label{ex-seq-2}
\]
where $\mathcal{F}$ is torsion free and moreover the singular locus of the quotient $Y$ of $X$ by the action induced by $D$ is the image of the subset of $X$ where $\mathcal{F}$ is not free~\cite{Ek87}. Let $Q=\mathcal{F}^{\ast\ast}/\mathcal{F}$. Then $Q$ has finite support and its support is exactly the isolated fixed points of $D$ since by Proposition~\ref{prop3} the singular locus of $Y$ is the set theoretic image of the isolated fixed points of $D$. Let $I_Q$ be the ideal sheaf of $Q$ in $X$. Tensoring the exact sequence
\[
0 \rightarrow I_Q \rightarrow \mathcal{O}_X \rightarrow Q \rightarrow 0 
\]
with $\mathcal{F}^{\ast\ast}$ we get the exact sequence
\[
0 \rightarrow \mathcal{F}^{\ast\ast}\otimes I_Q \rightarrow \mathcal{F}^{\ast\ast} \rightarrow Q \rightarrow 0.
\]
Therefore \[
\mathcal{F}=\mathcal{F}^{\ast\ast}\otimes I_Q.
\]
Since $X$ is smooth, $\mathcal{F}^{\ast\ast}$ is invertible and then the first part of the proposition follows.

Next we will show the second part of the lemma. It will moreover imply that the scheme structure of $Q$ is the same as the scheme structure of the embeded part and therefore they are the same as schemes and not only as sets.

Let $P\in X$ be an isolated singularity of $X$. Then locally in the \'etale topology $\mathcal{O}_X=k[x,y]$. Hence there are $f,g,h\in k[x,y]$ with $f,g$ relatively prime such that 
$D=h(f\partial/\partial x +g \partial/\partial y)$. Then the embeded part of $D$ is given by the ideal $(f,g)$ and the divisorial by $(h)$. Then map $\mathcal{O}_X(\Delta) \rightarrow T_X$ is given by
\[
\Phi \colon \mathcal{O}_X \rightarrow \mathcal{O}_X\frac{\partial}{\partial x} \oplus \mathcal{O}_X \frac{\partial}{\partial y}
\]
defined by $\Phi(1)=f\partial /\partial x +g \partial/ \partial y$. It is now easy to see that the map 
\[
\Psi \colon \mathcal{O}_X\frac{\partial}{\partial x} \oplus \mathcal{O}_X \frac{\partial}{\partial y} \rightarrow (f,g)
\]
given by $\Psi(F\partial/\partial x +G \partial /\partial y)= Gf-Fg$ induces an isomorphism between the cokernel $\mathrm{CoKer} (\Phi)$ and the ideal $(f,g)$. Therefore $I_Z= I_Q$. 

Finally, suppose that $D$ is of multiplicative type, i.e., $D^p=D$. Then by~\cite{R-S76}, in suitable choice of the local parameters $x$ and $y$, $D=a x\partial /\partial x + b y \partial /\partial y$, where $a,b \in \mathbb{F}_p$, the finite field of order $p$. Therefore at an isolated singular point of $D$, $I_Z=(x,y)$. Therefore if $D$ is of multiplicative type $Z$ is reduced and its length is equal to the number of isolated singular points of $D$. 
\end{proof}
\begin{remark}
The exact sequence in~\ref{ex-seq-1}.1 is not new. However I am not aware of any reference of it and also an explicit description of the relation between $Z$ and the embeded part of the fixed locus of $D$. This is the reason that I have included its proof here. 
\end{remark}
\begin{remark}
If $D$ is not of multiplicative type then the embeded part $Z$ of the fixed locus of $D$ may be nonreduced and its length strictly bigger than the number of isolated singular point of $D$, and hence of the number of singular points of the quotient $Y$. For example, let $A=k[x,y]$ and $D=x^2\partial/\partial x +y^2 \partial /\partial y$. If $p=2$ then $D^2=0$. $D$ has exactly one isolated singular point but the embeded part of the fixed locus is given by the ideal $(x^2,y^2)$ and therefore has length $4$.  
\end{remark}
\begin{proposition}\label{size-of-sing}
Let $X$ be a smooth surface defined over an algebraically closed field $k$ of characteristic $p>0$ . Let $D$ be a nontrivial global vector field on $X$ and let $I_Z $ be the ideal sheaf of the embeded part $Z$ or the fixed locus of $D$ and let $\Delta$ be its divisorial part. Then
\[
\mathrm{length}(\mathcal{O}_Z)=K_X \cdot \Delta +\Delta^2 +c_2(X).
\] 
Moreover, if $D$ is of multiplicative type ,i.e.,  $D^p=D$, $Z$ is reduced and then the number of isolated fixed points of $D$ is $K_X \cdot \Delta +\Delta^2 +c_2(X)$.
\end{proposition}

\begin{proof}

From Proposition~\ref{ex-seq-1} there is an exact sequence
\begin{gather}
0 \rightarrow \mathcal{O}_X(\Delta) \rightarrow T_X \rightarrow L \otimes I_Z \rightarrow 0.  \label{ex-seq-3}
\end{gather}
Since $Z$ has codimension 2 it follows that $L\cong \omega_X^{-1}\otimes \mathcal{O}_X(-\Delta)$. Moreover, from the  exact sequence
\[
0 \rightarrow L\otimes I_Z \rightarrow L \rightarrow \mathcal{O}_Z \rightarrow 0
\]
and the exact sequence~(\ref{ex-seq-3}) it follows that
\begin{gather*}
\mathrm{length}(\mathcal{O}_Z)=\chi(L)-\chi(L\otimes I_Z)=\chi(L)-\chi(T_X)+\chi(\mathcal{O}_X(\Delta))=\\
\chi(\mathcal{O}_X(\Delta))+\chi(\omega_X^{-1}\otimes \mathcal{O}_X(-\Delta))-\chi(T_X).
\end{gather*}
Be Riemann-Roch and Serre duality we get the following equalities.
\begin{gather*}
\chi(T_X)=2\chi(\mathcal{O}_X)+K_X^2-c_2(X),\\
\chi(\mathcal{O}_X(\Delta))=\chi(\mathcal{O}_X)+1/2(\Delta^2-\Delta \cdot K_X),\\
\chi(\omega_X^{-1}\otimes \mathcal{O}_X(-\Delta))= \chi(\omega_X^{2}\otimes \mathcal{O}_X(\Delta))=\\
\chi(\mathcal{O}_X)+1/2\left((\Delta+2K_X)^2-(\Delta+2K_X)\cdot K_X\right)=\\
\chi(\mathcal{O}_X)+1/2(2K_X^2+3K_X\cdot \Delta+\Delta^2).
\end{gather*}
Therefore from the above equations it follows that
\begin{gather*}
\mathrm{length}(\mathcal{O}_Z)=\chi(\mathcal{O}_X(\Delta))+\chi(\omega_X^{-1}\otimes \mathcal{O}_X(-\Delta))-\chi(T_X)=\\
2\chi(\mathcal{O}_X)+K_X^2+K_X\cdot \Delta +\Delta^2-\chi(T_X)=K_X\cdot \Delta +\Delta^2 +c_2(X),
\end{gather*}
as claimed. Suppose $D$ is of multiplicative type. Then by Lemma~\ref{ex-seq-1}, the embeded part $Z$ of the fixed locus of $D$ is reduced and its length is the same as the number of isolated singular points of $D$. Therefore the number of isolated singular points of $D$ is $K_X\cdot \Delta +\Delta^2 +c_2(X)$.
\end{proof}

\begin{corollary}\label{no-of-sing-of-quot}
Let $X$ be a smooth surface defined over an algebraically closed field $k$ of characteristic $p>0$ with a $\mu_p$ action. Let $\Delta$ be the divisorial part of the fixed locus of the action. Let $\pi \colon X \rightarrow Y$ be the quotient. Then $Y$ has exactly $K_X \cdot \Delta +\Delta^2 +c_2(X)$ singular points.
\end{corollary}
\begin{proof}
This follows immediately from Proposition~\ref{no-of-sing-of-quot}. Indeed, by Proposition~\ref{prop3}, the singular locus of $Y$ is exactly the set theoretic image of the embeded part of the fixed locus of the $\mu_p$ action. However, by Proposition~\ref{no-of-sing-of-quot} the number of isolated fixed points of $D$ is $K_X \cdot \Delta +\Delta^2 +c_2(X)$ and therefore this is the number of singular points of $Y$.
\end{proof}

Proposition~\ref{size-of-sing} suggests that $K_X \cdot \Delta$ is closely related with the size of the isolated singularities of $D$ and hence of the singular locus of the quotient $Y$. The next proposition shows that it decreases after blowing up a singular point of $D$.

\begin{proposition}\label{K-decreases}
Let $X$ be a smooth surface defined over a field of characteristic $p>0$. Let $D$ be a nonzero global vector field on $X$. Let $\Delta$ be its divisorial part. Let $f \colon X^{\prime} \rightarrow X$ be the blow up of an isolated singular point of $D$, $D^{\prime}$ the lifting of $D$ in $X^{\prime}$ and $\Delta^{\prime}$ its divisorial part. Then
\[
K_{X^{\prime}}\cdot \Delta^{\prime} \leq K_X \cdot \Delta.
\]
\end{proposition} 
\begin{proof}
Let $E$ be the $f$-exceptional curve. I will show that $\Delta^{\prime}=f^{\ast}\Delta +aE$, $a\geq 0$. Then \[
K_{X^{\prime}}\cdot \Delta^{\prime}=K_X\cdot \Delta -a \leq K_X\cdot \Delta.
\]
The proof of the previous claim will be by a direct local calculation of $D^{\prime}$. In suitable local coordinates at an isolated singular point of $D$, $\mathcal{O}_X=k[x,y]$ and $D$ is given by $D=h \left( f\partial/\partial x +g \partial /\partial y\right)$, where $f,g$ have no common factor. Locally at the standard open affine covers, the blow up is given by $\phi \colon k[x,y] \rightarrow k[s,t]$, $\phi(x)=s$, $\phi(y)=st$. Then it is easy to see that 
\[
D^{\prime}=h(s,st)\left(f(s,st)\frac{\partial}{\partial x}+\frac{1}{s}\left(tf(s,st)+g(s,st)\right)\frac{\partial}{\partial y}\right).
\]
It is now clear that $\Delta^{\prime}=f^{\ast}\Delta +aE$, $a\geq 0$.
\end{proof}

\begin{corollary}\label{sec3-cor-1}
Let $D$ be a nonzero global vector field of either multiplicative or additive type on a smooth surface $X$ defined over a field of characteristic 2. Let 
\[
\xymatrix{
X^{\prime}\ar[r]^f \ar[d]^{\pi^{\prime}} & X \ar[d]^{\pi} \\
Y^{\prime} \ar[r]^g & Y \\
}
\]
be the resolution of singularities of $D$ as in Proposition~\ref{prop4}. Let $\Delta$ be the divisorial part of $D$ and $\Delta^{\prime}$ the divisorial part of the lifting $D^{\prime}$ of $D$ on $X^{\prime}$. Then 
\begin{enumerate}
\item \[
K_{X^{\prime}}\cdot \Delta^{\prime} \leq K_X \cdot \Delta.
\]
\item
\[
K_{X^{\prime}}\cdot \Delta^{\prime}=4\left(\chi(\mathcal{O}_{X^{\prime}})-2\chi(\mathcal{O}_{Y^{\prime}})\right).
\]
\end{enumerate}
\begin{proof}
The proof of the first statement follows immediately from Proposition~\ref{K-decreases} since $f$ is a composition of blow ups of isolated singular points of $D$.

For the proof of the second statement, recall From Proposition~\ref{prop4} that $Y^{\prime}$ is the quotient of $X^{\prime}$ by the lifting $D^{\prime}$ of $D$ on $X^{\prime}$. Then by adjunction for purely inseperable maps~\cite{R-S76}, 
\begin{gather}\label{sec3-eq-10}
K_{X^{\prime}}=(\pi^{\prime})^{\ast}K_{Y^{\prime}}+\Delta^{\prime}.
\end{gather} 
Moreover, from Proposition~\ref{structure-of-pi} it follows (since $Y^{\prime}$ is smooth) that $\pi^{\prime} \colon X^{\prime} \rightarrow Y^{\prime}$ is a torsor. In particular $\pi_{\ast}\mathcal{O}_{X^{\prime}}$ fits in an exact sequence
\begin{gather}\label{sec3-eq-11}
0 \rightarrow \mathcal{O}_{Y^{\prime}} \rightarrow \pi_{\ast}\mathcal{O}_{X^{\prime}} \rightarrow M^{-1} \rightarrow 0,
\end{gather}
where $M=\mathcal{O}_{Y^{\prime}}(C^{\prime})$ is an invertible sheaf on $Y^{\prime}$. If the sequence splits then $D^2=D$ and if it doesn't split then $D^2=0$. Moreover, $K_{X^{\prime}}=(\pi^{\prime})^{\ast}(K_{Y^{\prime}}+C^{\prime})$. Therefore from~\ref{sec3-eq-10} we get that and $\Delta^{\prime}=(\pi^{\prime})^{\ast}C^{\prime}$. Then from~\ref{sec3-eq-11} we get that
\begin{gather}\label{sec3-eq-12}
\chi(M^{-1})=\chi(\pi^{\prime}_{\ast}\mathcal{O}_{X^{\prime}})-\chi(\mathcal{O}_{Y^{\prime}})=\chi(\mathcal{O}_{X^{\prime}})-\chi(\mathcal{O}_{Y^{\prime}}).
\end{gather}
From Riemann-Roch it follows that
\begin{gather}\label{sec3-eq-13}
\chi(M^{-1})=\chi(\mathcal{O}_{Y^{\prime}})+\frac{1}{2}((C^{\prime})^2+K_{Y^{\prime}}\cdot C^{\prime})=\\
\chi(\mathcal{O}_{Y^{\prime}})+\frac{1}{2}C^{\prime}\cdot (K_{Y^{\prime}}+C^{\prime})=\chi(\mathcal{O}_{Y^{\prime}})+\frac{1}{4}K_{X^{\prime}}\cdot \Delta^{\prime}
\end{gather}
Finally from~\ref{sec3-eq-12} amd ~\ref{sec3-eq-13} it follows that
\[
K_{X^{\prime}}\cdot \Delta^{\prime}=4\left(\chi(\mathcal{O}_{X^{\prime}})-2\chi(\mathcal{O}_{Y^{\prime}})\right),
\]
as claimed.
\end{proof}

\end{corollary}

\section{Surfaces with vector fields of multiplicative type in characteristic 2.}\label{sec-4}
The purpose of this section is to study smooth canonically polarized surfaces admitting global vector fields of multiplicative type. The main result is the following.
\begin{theorem}\label{mult-type}
Let $X$ be a smooth canonically polarized surface defined over an algebraically closed field $k$ of characteristic 2. Suppose that $X$ has a global vector field $D$ of multiplicative type. Then $\mu_2$ is a sub group scheme of $\mathrm{Aut}(X)$ and one of the following happens.
\begin{enumerate}
\item $K_X^2 \geq 5$.
\item $K_X^2=4$, $X$ is uniruled, $\pi_1^{et}(X)=\{1\}$ and $-2\leq \chi(\mathcal{O}_X)\leq 2$.
\item $K_X^2=3$, $X$ is uniruled, $\pi_1^{et}(X)=\{1\}$ and $-1\leq \chi(\mathcal{O}_X)\leq 1$.
\item $K_X^2=2$, $X$ is uniruled, $\pi_1^{et}(X)=\{1\}$ and $0\leq \chi(\mathcal{O}_X)\leq 1$.
\item $K_X^2=1$, $X$ is an algebraically simply connected unirational supersingular Godeaux surface. 
\end{enumerate}
Moreover, in all cases with $K_X^2<5$, $X$ is an inseparable quotient of degree 2 of a rational or ruled surface (possible singular) by a rational vector field. 
\end{theorem}

\begin{corollary}\label{sm-multi-type}
Let $X$ be a smooth canonically polarized surface defined over an algebraically closed field $k$ of characteristic 2. Suppose $K_X^2 <5$ and that one of the following happens.
\begin{enumerate}
\item $X$ is not uniruled.
\item $X$ is not simply connected, i.e., $\pi_1^{et}(X)\not=\{1\}$.
\item $K_X^2\in\{2,3\}$ and $\chi(\mathcal{O}_X)\geq 2$, or $K_X^2=4$ and $\chi(\mathcal{O}_X)\geq 3$.
\item $K_X^2=1$ and either
\begin{enumerate}
\item $\chi(\mathcal{O}_X) \geq 2$, or
\item $\chi(\mathcal{O}_X) =1$ and $X$ is either a classical or singular Godeaux surface.
\end{enumerate}
\end{enumerate}
Then $X$ does not have any nontrivial global vector fields of multiplicative type. Equivalently $\mathrm{Aut}(X)$ does not have a subgroup scheme isomorphic to $\mu_2$.
\end{corollary}

\begin{proof}[Proof of Theorem~\ref{mult-type}]
Suppose that $X$ has a nontrivial global vector field $D$ of multiplicative type. Then $D$ induces a nontrivial $\mu_2$ action on $X$. Let $\pi \colon X \rightarrow Y$ be the quotient. Then by Proposition~\ref{prop4}, $Y$ is normal and has only isolated surface singularities locally isomorphic to $xy+z^2=0$. Moreover, there is a commutative diagram
\begin{gather}\label{comm-diagram-1}
\xymatrix{
 & X^{\prime}\ar[r]^f \ar[d]^{\pi^{\prime}} & X \ar[d]^{\pi} \\
Z & Y^{\prime} \ar[l]_{\phi}\ar[r]^g & Y \\
}
\end{gather}
such that: $f$ is a resolution of the isolated singularities of $D$ through successive blow ups of its isolated singular points. $X^{\prime}$ is smooth, $D$ lifts to a vector field $D^{\prime}$ in $X^{\prime}$ with only divisorial singularities and $Y^{\prime}$ is the quotient of $X^{\prime}$ by the corresponding action of $\mu_2$. $Y^{\prime}$ is the minimal resolution of $Y$ and $Z$ its minimal model ($Y^{\prime}$ is the minimal resolution of $Y$ but may not be a minimal surface). Moreover, since the singularities of $Y$ are locally isomorphic to $xy+z^2=0$, $g$ is crepant. Hence $K_Y$ is Cartier and we have the following adjunction formulas
\begin{gather}\label{sec4-eq0}
K_{Y^{\prime}}=g^{\ast} K_Y \\
K_{Y^{\prime}}=\phi^{\ast}K_Z+F, \nonumber
\end{gather}
where $F$ is an effective $\phi$-exceptional divisor. Finally, by Proposition~\ref{prop3}, $2W$ is Cartier for any divisor $W$ of $Y$.

Let $\Delta$ be the divisorial part of the fixed locus of $D$. Then by Proposition~\ref{prop4}, $\Delta$ is smooth (perhaps disconnected), disjoint from the isolated singular points of $D$ and not an integral divisor of $D$. Therefore if $\Delta^{\prime}$ is the image of $\Delta$ in $Y$, $\Delta^{\prime}$ is in the smooth part of $Y$, $\pi_{\ast}\Delta =\Delta^{\prime}$ and $\pi^{\ast}\Delta^{\prime}=2\Delta$.

By adjunction for purely inseparable morphisms~\cite{Ek87}~\cite{R-S76},
\begin{gather}\label{sec4-eq1}
K_X=\pi^{\ast}K_Y+\Delta.
\end{gather}
Moreover, from Proposition~\ref{structure-of-pi}, $\pi$ is a torsor over a codimension 2 open subset of $Y$. In particular, there is a reflexive sheaf $L=\mathcal{O}_Y(C)$ on $Y$ such that $X=\mathrm{Spec} \left( \mathcal{O}_Y \oplus L^{-1} \right)$, and 
\begin{gather}\label{sec4-eq2}
K_X=\pi^{\ast}(K_Y+C).
\end{gather}
From this and~\ref{sec4-eq1} it follows that $\pi^{\ast}C=\Delta$. Then since $\pi_{\ast}\Delta = \Delta^{\prime}$ it follows that $\Delta^{\prime} \sim 2C$. Moreover, since $K_X$ is ample and $\pi$ a finite morphism, it follows that $K_Y+C$ is ample too.

Finally, from~\ref{sec4-eq2} and the fact that $\pi$ is finite of degree 2 it follows that
\begin{gather}\label{sec4-eq3}
K_X^2=2(K_Y+C)^2=2K_Y\cdot (K_Y+C) +2C\cdot (K_Y+C).
\end{gather}

The proof will be in several steps, according to the Kodaira dimension $k(Y)$ of $Y$.

\textbf{Case 1.} Suppose $k(Y)=2$. In this case I will show that $K_Y^2 \geq 5$.

In this case take $Z$ in the commutative diagram~\ref{comm-diagram-1} to be the canonical model of $Y^{\prime}$ and not simply its minimal model. Canonical models of smooth surfaces exist in any characteristic by~\cite{Art62}. 

If $Y^{\prime}$ is minimal then $F=0$, otherwise not. Moreover, if $F\not= 0$, then $g_{\ast}F \not= 0$ since $g$ does not contract $-1$ curves. 

Since $ 2C\sim \Delta^{\prime}$ which is Cartier and effective, it follows that 
\begin{gather}\label{sec4-eq4}
2C\cdot (K_Y+C)\geq 1.
\end{gather}
 Also,
\begin{gather}\label{sec4-eq5}
K_Y\cdot (K_Y+C)=g^{\ast}K_Y \cdot g^{\ast}(K_Y+C)=\\
K_{Y^{\prime}}\cdot g^{\ast}(K_Y+C)=(\phi^{\ast}K_Z +F)\cdot g^{\ast}(K_Y+C).\nonumber
\end{gather}

Suppose $F\not= 0$ and $C\not=0$. Then $F\cdot g^{\ast}(K_Y+C)=g_{\ast}F \cdot (K_Y+C)>0$, since $K_Y+C$ is ample and $g_{\ast}F\not= 0$. Moreover, $\phi^{\ast}K_Z \cdot g^{\ast}(K_Y+C) >0$. Indeed, take $n>>0$ such that $n(K_Y+C)$ is very ample. Then there is a divisor $H\in |n(K_Y+C)|$ such that $g^{\ast}H$ is not contracted by $\phi$. Since $K_Z$ is ample then $\phi^{\ast}K_Z \cdot g^{\ast}(K_Y+C) >0$, as claimed. Hence $K_Y\cdot (K_Y+C) \geq 2$. Then from~\ref{sec4-eq2},~\ref{sec4-eq4} we get that $K_X^2 \geq 5$, as claimed.

Suppose that $C\not= 0$ but $F=0$, i.e., $Y^{\prime}$ is minimal. Then $K_Y^2=K_{Y^{\prime}}^2\geq 1$. Moreover $K_Y\cdot C=1/2(K_Y \cdot \Delta^{\prime} )\geq 0$. Suppose that $K_Y \cdot \Delta^{\prime}=0$. Since $\Delta^{\prime}$ is in the smooth part of $Y$ it follows that $\Delta^{\prime\prime}=g^{\ast}\Delta^{\prime}\cong \Delta^{\prime}$. Then 
$K_{Y^{\prime}}\cdot \Delta^{\prime\prime}=0$. Since $Y^{\prime}$ is minimal of general type it follows that every irreducible component of $\Delta^{\prime\prime}$ is a smooth rational curve of self intersection $-2$. Therefore the same holds for $\Delta^{\prime}$. Let now $W^{\prime}$ be an irreducible component of $\Delta^{\prime}$. Then $\pi^{\ast}W^{\prime}=2W$, where $W$ is an irreducible component of $\Delta$. Then $4W^2=2(W^{\prime})^2=-4$ and hence $W^2=-1$. But then $K_X\cdot W=-1 <0$, which is impossible since $K_X$ is ample. Therefore $K_Y \cdot C >0$, and in fact $\geq 1$ since $K_Y$ is Cartier. Then again from equation~\ref{sec4-eq2} it follows that $K_X^2\geq 5$.

Suppose now that $C=0$. This implies that $\Delta =0$ and the fixed locus of $D$ does not have a divisorial part. This case can only happen if $k(Y)=2$. Indeed, since $\Delta=0$ it follows that $K_X=\pi^{\ast}K_Y$ and therefore since $\pi$ is finite, $K_Y$ is ample. Moreover, since $Y$ has singularities of type $A_1$, $K_{Y^{\prime}}=g^{\ast}K_Y$. Hence $K_{Y^{\prime}}$ is nef and big and hence $k(Y^{\prime})=2$.

Since $\Delta=0$, $D$ has no divisorial part. However there may be isolated singular points of $D$. Then from Proposition~\ref{size-of-sing}, $Y$ has exactly $c_2(X)$ singular points. If $c_2(x)=0$ then $K_X^2=12\chi(\mathcal{O}_X) \geq 12$. Suppose that $c_2(X)>0$. Then $c_2(X)=\chi_{et}(Y)$. Moreover,
\[
\chi_{et}(Y^{\prime})=\chi_{et}(Y)+\sum(\chi_{et}(E_i)-1)
\]
where $E_i$ are the reduced connected components of the $g$-exceptional locus. Since $Y$ has singularities of type $xy+z^2=0$, the $g$-exceptional curves are exactly $c_2(X)$ number isolated $(-2)$-curves. Hence $\chi_{et}(Y^{\prime})=2c_2(X)$.

Since $\Delta=0$, $K_X=\pi^{\ast}K_Y$. Hence $K_X^2=2K_Y^2=2K_{Y^{\prime}}^2$=2d. Then from Noether's formula $c_2(X)=12\chi(\mathcal{O}_X)-2d$. Moreover, again from Noether's formula on $Y^{\prime}$,  
\[
12\chi(\mathcal{O}_{Y^{\prime}}) =d+c_2(Y^{\prime})=d+2c_2(X)=-3d+24\chi(\mathcal{O}_X).
\]
 However, this relation is impossible for $d=1$ or $d=2$. Hence in all cases $K_X^2 \geq 5$, as claimed.

\textbf{Case 2.} Suppose $k(Y)=1$. In this case I will show again that $K_Y^2 \geq 5$.

First notice that as explained in the proof of Case 1., $\Delta \not= 0$ in this case. 

The minimal model $Z$ of $Y^{\prime}$ is then a minimal surface of Kodaira dimension 1. Then by the classification of surfaces~\cite{BM76}~\cite{BM77}~\cite{Ba01}, $Z$ admits either an elliptic or quasielliptic fibration, i.e., there is a fibration $\psi \colon Z \rightarrow B$, where $B$ is a smooth curve and the fibers of $\psi$ have arithmetic genus 1. Also from the classification of surfaces, there is $n>>0$ such that $nK_Z=\psi^{\ast}(W)$, where $W$ is a positive divisor in $B$.

Next I will show that $\phi^{\ast}K_Z\cdot E =0$, where $E$ is any $g$-exceptional curve. Indeed,
\[
\phi^{\ast}K_Z \cdot E=K_{Y^{\prime}}\cdot E -F\cdot E=g^{\ast}K_Y \cdot E -F\cdot E =-F\cdot E.
\]
Since $K_Z$ is nef, it follows that $\phi^{\ast}K_Z \cdot E=K_Z\cdot \phi_{\ast}E \geq 0$. Therefore, $F\cdot E \leq 0$. If it is strictly negative, then $E\subset F$ and hence $E$ is $\phi$-exceptional. But then $\phi^{\ast}K_Z\cdot E=0$. Hence in any case $\phi^{\ast}K_Z \cdot E=0$ for any $g$-exceptional curve. Therefore $\phi^{\ast}K_Z=g^{\ast}H$, where $H$ is a Cartier divisor on $Y$. Moreover, since $nK_Z$ is positive for large enough $n$, it follows that $nH$ is positive too. Therefore from~\ref{sec4-eq0} it follows that
\[
K_Y=H+g_{\ast}F.
\]
Since $K_Y$ and $H$ are Cartier, $g_{\ast}F$ is an effective Cartier divisor as well. Then from~\ref{sec4-eq2} it follows that
\begin{gather}\label{sec4-eq5}
K_X^2=2(K_Y+C)^2=2K_Y\cdot (K_Y+C) +2C\cdot (K_Y+C)=\\ 
2H \cdot (K_Y+C)+2g_{\ast}F\cdot (K_Y+C) +2C\cdot(K_Y+C).\nonumber
\end{gather}
Suppose that $g_{\ast}F\not= 0$, i.e., $Y^{\prime}$ is not a minimal surface. Then, since $2C$ is Cartier and equivalent to an effective divisor, the above equation implies that $K_X^2 \geq 5$, as claimed.

Suppose that $g_{\ast}F=0$, hence $Y^{\prime}$ is a minimal surface. Then in any case, equation~\ref{sec4-eq3} shows that $K_X^2 \geq 3$ and $K_X^2 \geq 5$ unless $K_Y \cdot (K_Y+C)=1$ and $C \cdot (K_Y+C)=1/2$. Suppose that this is the case. Then considering that $\pi^{\ast}C=\Delta$ and that $K_Y^2=K_{Y^{\prime}}^2=0$, it follows that $K_X\cdot \Delta =1$ and $\Delta^2 =-1$. 

Let $N_{fix}$ be the number of isolated fixed points of $D$. Then by Corollary~\ref{no-of-sing-of-quot}, 
\[
N_{fix}=K_X\cdot \Delta +\Delta^2 +c_2(X)=c_2(X).
\]
Hence $Y$ has exactly $c_2(X)$ singular points, all of them of type $A_1$. Such singularities are resolved by a single blow up. Hence $f$ is the composition of $c_2(X)$ blow ups. Hence from diagram~\ref{comm-diagram-1} it follows that
\[
c_2(Y^{\prime})=c_2(X^{\prime})=c_2(X)+c_2(X)=2c_2(X).
\]
Noether's formula (true also in characteristic 2) gives that
\[
12\chi(\mathcal{O}_{Y^{\prime}})=K_{Y^{\prime}}^2+c_2(Y^{\prime})=2c_2(X)
\]
and therefore $c_2(X)=6d$. However, Noether's formula for $X$ gives that
\[
12\chi(\mathcal{O}_X) =K_X^2 +c_2(X)=3+6d,
\]
which is clearly impossible.

\textbf{Case 2.} Suppose $k(Y)=0$. In this case I will show that;
\begin{enumerate}
\item If $K_X^2 \leq 4$, then $X$ is algebraically simply connected, unirational and $\chi(\mathcal{O}_X)=1$. Moreover, $X$ is an inseparable quotient of degree 2 of a rational or ruled surface (possible singular) by a rational vector field. 
\item If $K_X^2=1$, then $X$ is an algebraically simply connected unirational supersingular Godeaux surface. 
\end{enumerate}
Again for the same reasons as in the case when $k(Y)=1$, $\Delta\not=0$.

Suppose that $F\not=0$, i.e., $Y^{\prime}$ is not minimal. Its minimal model $Z$ is a minimal surface of Kodaira dimension zero. Therefore $12K_Z=0$. Hence 
\[
K_{Y^{\prime}}=\phi^{\ast}K_Z+F \equiv F
\]
and therefore \[
K_Y\equiv g_{\ast}F=F^{\prime}
\]
where $F^{\prime}$ is a nonzero effective Cartier divisor. Hence \[
K_Y\cdot (K_Y+C)=F^{\prime}\cdot (K_Y+C) \geq 1
\]
since $ K_Y+C$ is ample. Moreover, as explained in Case 1, $2C\cdot (K_Y+C)\geq 1$. Therefore from~\ref{sec4-eq3} we get that 
\[
K_X^2=2(K_Y+C)^2=2K_Y\cdot (K_Y+C) +2C\cdot (K_Y+C)\geq 3.
\]
In fact, $K_X^2 \geq 5$ unless 
\begin{gather}\label{sec4-eq6}
K_Y\cdot (K_Y+C)=1\\
C\cdot (K_Y+C)=1/2,\nonumber
\end{gather}
in which case $K_X^2=3$. I will show that this case is impossible. Suppose that this is the case. Then 

\textbf{Claim 1.} 
\begin{enumerate}
\item $\Delta \cong \mathbb{P}^1$.
\item $\Delta^2=-3$, $K_X \cdot \Delta = 1$ and $K_{Y^{\prime}}^2=-1$.
\end{enumerate}
Indeed. From~\ref{sec4-eq6} it follows that 
\[
K_X \cdot \Delta =\pi^{\ast}(K_Y+C)\cdot \pi^{\ast}C=2(K_Y+C)\cdot C =1.
\]
Therefore, since $K_X$ is ample, $\Delta$ is an irreducible and reduced curve. On the other hand the equation $K_Y\cdot (K_Y+C)=1$ gives that
\[
\pi^{\ast}K_Y \cdot \Delta = 2K_Y \cdot C=2(1-K_Y^2)=2(1-K_{Y^{\prime}}^2)=2(1-F^2)\geq 4,
\]
since $F^2<0$. Now again from~\ref{sec4-eq6} we get that
\[
\Delta^2=1-\Delta\cdot \pi^{\ast}K_Y \leq -3.
\]
However the genus formula for $\Delta$ gives that
\[
\mathrm{p}_a(\Delta)=\frac{1}{2}(2+K_X\cdot \Delta+\Delta^2)=\frac{1}{2}(3+\Delta^2)\leq 0,
\]
since $\Delta^2 \geq -3$. This implies that $\Delta^2=-3$ and $\mathrm{p}_a(\Delta)=0$. Hence $\Delta \cong \mathbb{P}^1$. Finally from the relation $K_X=\pi^{\ast}K_Y+\Delta$ it follows that \[
K_Y^2=\frac{1}{2}(K_X-\Delta)^2=K_X^2+\Delta^2-2K_X\cdot \Delta =-1,
\]
and the claim is proved.

\textbf{Claim 2.} The map $\phi$ is a single blow up. 

Indeed. From Claim 1 it follows that $K_{Y^{\prime}}^2=K_Y^2=-1$. On the otherhand, $K_Z^2=0$ and $\phi$ is a composition of blow ups. Considering that $K^2$ is reduced by 1 after every blow up it follows that $\phi$ is a single blow up.

Let $N_{fix}$ be the number of isolated singular points of $D$. Then from Corollary~\ref{no-of-sing-of-quot} and Claim 1,
\[
N_{fix}=K_X\cdot \Delta +\Delta^2+c_2(X)=-2+c_2(X).
\]
$Y^{\prime}$ is the minimal resolution of $Y$ whic has exactly $N_{fix}$ singular points, all of type $A_1$. Hence the $g$-exceptional curves are  exactly $N_{fix}$ isolated $-2$ curves. Then from diagram~\ref{comm-diagram-1} and Claim 2 we get that \[
c_2(Z)=c_2(Y^{\prime})-1=\chi_{et}(Y^{\prime})+N_{fix}-1=c_2(X)+N_{fix}-1=2c_2(X)-3.
\]
By the classification of surfaces~\cite{BM76}~\cite{BM77}, $c_2(Z)\in \{0,12,24\}$. Hence $2c_2(X)\in\{ 3, 15, 27\}$, which is clearly impossible. Hence the case $K_X^2=3$ is impossible and therefore, if $Y^{\prime}$ is not minimal, then $K_X^2 \geq 5$.

Assume now that $Y^{\prime}$ is minimal. Then $K_{Y^{\prime}}\equiv 0$ and hence $K_Y \equiv 0$ which implies that $K_X\equiv \Delta$. Then 
\begin{gather}\label{sec4-eq8}
N_{fix}=K_X\cdot \Delta +\Delta^2 +c_2(X)=2K_X^2+c_2(X).
\end{gather}
This implies that $Y$ is singular. Indeed. If $Y$ was smooth, then $Y=Y^{\prime}$ and $c_2(X)=c_2(Y^{\prime}) \in\{0,12,24\}$, since $Y^{\prime}$ is a minimal surface of Kodaira dimension zero. In particular, $c_2(X)\geq 0$ and hence $N_{fix}>0$. Hence $Y$ is singular and so $Y^{\prime}\not= Y$. Moreover,
\begin{gather}\label{sec4-eq7}
c_2(Y^{\prime})=\chi_{et}(Y)+N_{fix}=c_2(X)+N_{fix}=2(K_X^2+c_2(X))=24\chi(\mathcal{O}_X),
\end{gather}
from Noether's formula. Now since $c_2(Y^{\prime})\in \{0,12,24\}$ it follows from~\ref{sec4-eq7} that $c_2(Y^{\prime})=0$ or $24$ and consequently $\chi(\mathcal{O}_X)=0$ or $1$. Hence $Y^{\prime}$ is either an Abelian, a K3 or a quasi-hyperelliptic surface~\cite{BM76}~\cite{BM77}. I will show that $Y^{\prime}$ is a K3 surface. Suppose not. Then $\chi(\mathcal{O}_X)=0$, $c_2(Y^{\prime})=0$ and $Y^{\prime}$ is either an Abelian surface or a quasi-hyperelliptic surface. Since $\chi(\mathcal{O}_X)=0$, Noether's formula gives that $c_2(X)=-K_X^2 <0$, since $K_X$ is ample. Hence from~\cite{SB91} it follows that $X$ is uniruled and therefore so is $X^{\prime}$. But a uniruled surface cannot dominate an Abelian surface. Hence this case is impossible and therefore $Y^{\prime}$ is a quasi-hyperelliptic surface. Then there exists an elliptic or quasi-hyperelliptic fibration $\Phi \colon Y^{\prime}\rightarrow E$, where $E=\mathrm{Alb}(Y^{\prime})$ is a smooth elliptic curve. Moreover every fiber or $\Phi$ is 
irreducible~\cite{BM76}. if $Y$ is singular then the $g$-exceptional curves are isolated smooth rational $-2$ curves. Since $E$ is elliptic, every $g$-exceptional curve must contract to a point by $\Phi$. But this is impossible because every fiber of $\Phi$ is an irreducible curve or arithmetic genus 1. Hence $Y$ is a K3 surface. Therefore taking into consideration that the \'etale fundamental groupd is a birational invariant and that $\pi^{\prime}$ gives an equivalence between the \'etale sites of $X^{\prime}$ and $Y^{\prime}$ we get that 
\[
\pi_1^{et}(X)=\pi_1^{et}(X^{\prime})=\pi_1^{et}(Y^{\prime})=\{1\}
\]
and therefore $X$ is algebraically simply connected.

Next I will show that $Y^{\prime}$ is unirational. In order to show this I will show that $Y$ has at least 13 singular points of type $A_1$. Then, $Y^{\prime}$ has at least 13 isolated $-2$ curves and hence in the terminology of~\cite{SB96}, $Y^{\prime}$ has a special configuration $E$ of rank at least 13. Therefore it is unirational~\cite{SB96}. Considering that $\chi(\mathcal{O}_X)=1$, from~\ref{sec4-eq8} we get that
\[
N_{fix}=2K_X^2+c_2(X)=K_X^2+12\chi(\mathcal{O}_X)=K_X^2+12\geq 13,
\]
since $K_X^2>0$. Hence $Y$ has at least 13 singular points (all necessarily of type $A_1$), as claimed.

Next I will show that $X$ is a purely inseparable quotient of degree 2 of a unirational surface. Let $F \colon Y_{(2)}\rightarrow Y$ be the $k$-linear Frobenious. Then there is a factorization
\[
\xymatrix{
Y_{(2)}\ar[dd]^{F}\ar[dr]^{\nu} &  \\
                                    & X \ar[dl]^{\pi}\\
Y &
}
\]
Since $Y$ is unirational, $Y_{(2)}$ is unirational too. Moreover, $\nu$ is purely inseparable of degree 2 and in addition any such map is induced by a rational vector field on $Y_{(2)}$~\cite{R-S76}. 

Suppose now that $K_X^2=1$. Since we already proved that $\chi(\mathcal{O}_X)=1$, $X$ is a numerical Godeaux surface. It remains to show that $X$ is a supersingular Godeaux. This means that $h^1(\mathcal{O}_X)=1$ and map $F^{\ast} \colon H^1(\mathcal{O}_X) \rightarrow H^1(\mathcal{O}_X)$ induced by the Frobenious is zero.

Since $\chi(\mathcal{O}_X)=1$ it follows that $h^1(\mathcal{O}_X)=p_g(X)=h^0(\omega_X)$. However, since $Y^{\prime}$ is a K3 surface, $\omega_{Y^{\prime}}\cong \mathcal{O}_{Y^{\prime}}$ and therefore $\omega_Y\cong \mathcal{O}_Y$. Then from~\ref{sec4-eq2} it follows that $\omega_X=(\pi^{\ast}L)^{\ast\ast}$, and therefore 
\[
H^0(\omega_X)=H^0(\pi_{\ast}(\pi^{\ast}L)^{\ast\ast})=H^0((L\otimes (\mathcal{O}_Y\oplus L^{-1})^{\ast\ast})=H^0(\mathcal{O}_Y\oplus L)\not=0.
\]
Hence $X$ is either a singular or a supersingular Godeaux surface. If $X$ was singular, $F^{\ast}$ is injective and the exact sequence~\cite[Page 127]{Mi80}
\[
0 \rightarrow \mathbb{Z}/2\mathbb{Z} \rightarrow \mathcal{O}_X \stackrel{F-1}{\rightarrow} \mathcal{O}_X \rightarrow 0
\]
gives that $H^1_{et}(X,\mathbb{Z}/2\mathbb{Z})\not= 0$. Hence $X$ has \'etale $2$-covers. But this is impossible since $\pi_1^{et}(X)=\{1\}$. Hence $X$ is a supersingular Godeaux surface.

\textbf{Case 4.} Suppose that $k(Y)=-\infty$. In this case I will show that $K_X^2 \geq 2$, $X$ is uniruled and $\pi_1^{et}(X)=\{1\}$. Moreover, 
\begin{enumerate}
\item If $K_X^2=4$, then $-2\leq \chi(\mathcal{O}_X)\leq 2$.
\item If $K_X^2=3$, then $-1\leq \chi(\mathcal{O}_X)\leq 1$.
\item If $K_X^2=2$, then $0\leq \chi(\mathcal{O}_X)\leq 1$.
\end{enumerate}

For the proof I will need the following simple result. Its statement about the case $K_X^2=1$ can be found in~\cite{Li09}.

\begin{lemma}\label{b1}
Let $X$ be a smooth surface of general type defined over an algebraically closed field $k$. Suppose that one of the following happens.
\begin{enumerate}
\item $K_X^2=1$.
\item $K_X^2\in\{2,3\}$ and $\chi(\mathcal{O}_X) \geq 2$.
\item $K_X^2=4$ and $ \chi(\mathcal{O}_X)\geq 3$.
\end{enumerate}
Then $b_1(X)=0$. Moreover,
\begin{enumerate}
\item If $K_X^2=1$ then $|\pi_1^{et}(X)|\leq 6$.
\item If $K_X^2=2$ then $|\pi_1^{et}(X)|\leq 3$. Moreover, if $\chi(\mathcal{O}_X) \geq 3$, then $\pi_1^{et}(X)=\{1\}$.
\item If $K_X^2=3$ then $|\pi_1^{et}(X)|\leq 6$. Moreover, if $\chi(\mathcal{O}_X) \geq 4$, then $\pi_1^{et}(X)=\{1\}$.
\item If $K_X^2=4$ then $|\pi_1^{et}(X)|\leq 3$. Moreover, if $\chi(\mathcal{O}_X) \geq 4$, then $\pi_1^{et}(X)=\{1\}$.
\end{enumerate}
\end{lemma}

\begin{proof}[Proof of the lemma.]
It is well known that for all but finitely many primes $l$,
\begin{gather}\label{sec4-eq9}
H^1_{et}(X,\mathbb{Z}/l\mathbb{Z})\cong\left(\mathbb{Z}/l\mathbb{Z}\right)^{b_1(X)}.
\end{gather}
The claim then that $b_1(X)=0$ will follow from the finiteness of $\pi_1^{et}(X)$. Let $f \colon Y \rightarrow X$ be an \'etale cover of degree $n$. Then $K_Y^2=nK_X^2$ and $\chi(\mathcal{O}_Y)=n\chi(\mathcal{O}_X)$. From Noether's inequality we get that
\[
nK_X^2=K_Y^2\geq 2\mathrm{p}_g(Y)-4=2(\chi(\mathcal{O}_Y)-1+h^1(\mathcal{O}_Y))-4\geq 2n\chi(\mathcal{O}_X)-6.
\]
Therefore,
\[
n(2\chi(\mathcal{O}_X)-K_X^2)\leq 6.
\]
Hence if $2\chi(\mathcal{O}_X)-K_X^2>0$, then 
\begin{gather}\label{sec4-eq10}
n \leq \frac{6}{2\chi(\mathcal{O}_X)-K_X^2}\leq 6
\end{gather}

By~\cite[Corollary 1.8]{Ek87}, $\chi(\mathcal{O}_X) \geq 2-K_X^2$. Hence if $K_X^2=1$, then $2\chi(\mathcal{O}_X)-K_X^2>0$. Also under the hypotheses of~\ref{b1}.2,3, $\chi(\mathcal{O}_X)-K_X^2>0$ as well. Therefore in all these cases, $n \leq 6$ and hence  $|\pi_1^{et}(X)|\leq 6$ and therefore $b_1(X)=0$. The statements about the orders of the e\'tale fundamental group follow by bounding the values of $n$ from~\ref{sec4-eq10} in each case. 
\end{proof}

Suppose that $k(Y)=-\infty$. Then $Z$ is ruled surface over a smooth curve $B$ and therefore $Y$ is ruled and hence $X$ is ruled as well. Let $\Psi \colon Z \rightarrow B$ be the ruling. Considering that the \'etale fundamental group is invariant under birational maps between smooth varieties and under purely inseparable maps, we get that 
\[
\pi_1^{et}(X)=\pi_1^{et}(X^{\prime})=\pi_1^{et}(Y^{\prime})=\pi_1^{et}(Z)=\{1\},
\]
since $Z$ is ruled. 

In order to conclude the proof of Case 4, it remains to show that the case $K_X^2=1$ does not happen and the claim about the euler characteristics. 

From Lemma~\ref{b1} it follows that if $K_X^2=1$ then $b_1(X)=0$. Moreover  if the inequalities for the euler characteristics in the statement  of Case 4 do not hold, then $b_1(X)=0$ as well. Hence all the claims will follow if I show that if $b_1(X)=0$ then $X$ does not have nontrivial global vector fields. 

Suppose then that $b_1(X)=0$. Then $B \cong \mathbb{P}^1_k$ and hence both $Y^{\prime}$, $Y$ are rational and therefore $X$ is unirational. 

\textbf{Claim:} $X$ lifts to $W_2(k)$, where $W_2(k)$ is the ring of second Witt vectors over $k$.

Suppose that the claim is true. Then from Corollary~\ref{lifts-to-zero}, $X$ has no nontrivial global vector fields.

It remains then to prove the claim. Recall that $Y^{\prime}$ is the quotient of $X^{\prime}$ by the $\mu_2$ action on $X^{\prime}$ induced by the lifting $D^{\prime}$ of $D$ on $X^{\prime}$. Since both $X^{\prime}$ and $Y^{\prime}$ are smooth it follows from Proposition~\ref{structure-of-pi} that 
\[
X^{\prime}=\mathrm{Spec}\left( \mathcal{O}_{Y^{\prime}}\oplus M^{-1} \right)
\]
where $M$ is an invertible sheaf on $Y^{\prime}$ and the ring structure on $\mathcal{O}_{Y^{\prime}}\oplus M^{-1}$ is induced by a section $s$ of $M^2$. It is well known~\cite{Ha10} that $H^2(T_{Y^{\prime}})$ is an obstruction space for deformations of $X$ over local Artin rings and $H^1(\mathcal{O}_{Y^{\prime}})$, $H^2(\mathcal{O}_{Y^{\prime}})$ are obstruction spaces for deformations of line bundles and sections of line bundles on $Y^{\prime}$. Since $Y^{\prime}$ is rational, all these spaces are zero. Therefore, $Y^{\prime}$, $M$ and the section $s \in H^0(M^2)$ all lift compatively to $W_2(k)$. Let $Y^{\prime}_2$, $M_2$ and $s_2$ be the liftings of $Y^{\prime}$, $M$ and $s$, respectively. Then
\[
X^{\prime}_2=\mathrm{Spec}\left( \mathcal{O}_{Y_2^{\prime}}\oplus M_2^{-1}\right)
\]
is a lifting of $X^{\prime}$ over $W_2(k)$. Next I will show that $X$ lifts over $W_2(k)$ too. Let $X_2$ be the ringed space $(X,f_{\ast}\mathcal{O}_{X^{\prime}_2})$. Then $X_2$ is a deformation of $X$ over $W_2(k)$. Indeed. Recall~\cite{EV92} that $W_2(K)$ is $k\oplus k\cdot p$ as an additive group, with ring structure given by
\[
(x_1+y_1\cdot p)(x_2+y_2\cdot p)=x_1x_2+(x_1y_2+x_2y_1)\cdot p.
\]
Then one can easily see that there is an exact sequence
\[
0 \rightarrow k \stackrel{\sigma}{\rightarrow} W_2(k) \rightarrow k \rightarrow 0,
\]
where $\sigma(x)=x\cdot p$. Then tensoring with $\mathcal{O}_{X^{\prime}_2}$ we get the exact sequence
\[
0 \rightarrow \mathcal{O}_{X^{\prime}} \stackrel{\sigma}{\rightarrow} \mathcal{O}_{X^{\prime}_2} \rightarrow \mathcal{O}_{X^{\prime}} \rightarrow 0
\]
Applying $f_{\ast}$ we get the exact sequence
\[
0 \rightarrow f_{\ast}\mathcal{O}_{X^{\prime}} \rightarrow f_{\ast}\mathcal{O}_{X^{\prime}_2} \rightarrow f_{\ast}\mathcal{O}_{X^{\prime}} \rightarrow R^1f_{\ast}\mathcal{O}_{X^{\prime}}
\]
Considering that $f_{\ast}\mathcal{O}_{X^{\prime}}\cong \mathcal{O}_X$ and that $R^1f_{\ast}\mathcal{O}_X=0$, we get the following exact sequence
\[
0 \rightarrow \mathcal{O}_X \stackrel{p}{\rightarrow} f_{\ast}\mathcal{O}_{X_2^{\prime}} \rightarrow \mathcal{O}_X \rightarrow 0.
\]
Now tensoring with $k$ over $W_2(k)$ we get that $\mathcal{O}_{X^{\prime}_2}\otimes_{W_2(k)}k\cong \mathcal{O}_{X}$. Finally from the infinitesimal criterion of flatness, $X_2$ is flat over $W_2(k)$ and hence $X_2$ is a deformation of $X$ over $W_2(k)$, as claimed.

\end{proof}
\section{Surfaces with vector fields of addditive type in characteristic 2.}\label{sec-5}

The purpose of this section is to study smooth canonically polarized surfaces which admit nontrivial global vector fields of additive type. This case is more complicated  from the multiplicative case essentially because $\alpha_2$ actions are harder to describe than $\mu_2$ actions. One of the  difficulties is that the singularities of the quotient of the surface with the induced $\alpha_2$ action are more complicated than those that appear in the multiplicative case. In fact, not only they are not necessarily canonical, but they may not even be rational. However, from Proposition~\ref{prop4}, they are Gorenstein. 

The main result of this section is the following.

\begin{theorem}\label{additive-type}
Let $X$ be a smooth canonically polarized surface defined over an algebraically closed field of characteristic 2. Suppose that $X$ has a nontrivial global vector field of additive type. Then $\alpha_2$ is a sub group scheme of $\mathrm{Aut}(X)$ and one of the following happens.
\begin{enumerate}
\item $K_X^2 \geq 3$.
\item $K_X^2=2$ and $X$ is uniruled. Moreover, if $\chi(\mathcal{O}_X)\geq 2$, then $X$ is unirational and $\pi_1^{et}(X)=\{1\}$.
\item $K_X^2=1$, $\pi_1^{et}(X)=\{1\}$, $p_g(X) \leq 1$ and $X$ is unirational. 
\end{enumerate}
Moreover, if $1 \leq K_X^2\leq 2$, then $X$ is the quotient of a ruled or rational surface (maybe singular) by a rational vector field.
\end{theorem}

From the following theorem it immediately follows that.

\begin{corollary}\label{cor-additive-type}
Let $X$ be a smooth canonically polarized surface defined over an algebraically closed field of characteristic 2. Suppose that either $K_X^2=2$ and $X$ is not uniruled or $K_X^2=1$ and one of the following happens
\begin{enumerate}
\item $p_g(X)=2$.
\item $\chi(\mathcal{O}_X) =3$.
\item $X$ is a singular Godeaux (this means that the action of the Frobenious on $H^1(\mathcal{O}_X)$ is zero).
\item $\pi_1^{et}(X) \not= \{1\}$.
\item $X$ is not unirational.
\end{enumerate}
Then $X$ has no nontrivial global vector field of additive type. In particular, $\alpha_2$ is not a subgroup scheme of $\mathrm{Aut}(X)$.
\end{corollary}

\begin{proof}[Proof of Theorem~\ref{additive-type}]
The proof will be along the lines of the proof of Theorem~\ref{mult-type}. However, there are many differences between the two cases that complicate things. Suppose that $X$ admits a nontrivial global vector field $D$ of additive type.  Then $D$ induces a nontrivial $\alpha_2$ action on $X$. Let $\pi \colon X \rightarrow Y$ be the quotient. Then by Proposition~\ref{prop4}, $Y$ is normal and its local class groups are 2-torsion. Moreover, there is a commutative diagram
\begin{gather}\label{sec5-diagram-1}
\xymatrix{
 & X^{\prime\prime}\ar[r]^{f^{\prime\prime}} \ar[d]_{\pi^{\prime\prime}} & X \ar[d]^{\pi} \\
 & Y^{\prime\prime} \ar[r]^{g^{\prime\prime}} & Y \\
}
\end{gather}
such that: $f^{\prime\prime}$ is a resolution of the isolated singularities of $D$ through successive blow ups of its isolated singular points. $X^{\prime\prime}$ is smooth, $D$ lifts to a vector field $D^{\prime\prime}$ in $X^{\prime\prime}$ with only divisorial singularities and $Y^{\prime\prime}$ is the quotient of $X^{\prime\prime}$ by the corresponding action of $\alpha_2$. However, unlike the multiplicative case the singularities of $Y$ are not necessarily canonical and $Y^{\prime\prime}$ may not be the minimal resolution of $Y$. 

Let $g \colon Y^{\prime}\rightarrow Y$ be the minimal resolution. Then there is a factorization
\[
\xymatrix{
     &  Y^{\prime}\ar[dr]^{g} & \\
Y^{\prime\prime}\ar[ur]^{g^{\prime}} \ar[rr]^{g^{\prime\prime}}  & & Y
}
\]
Let $F$ be the reduced $g^{\prime}$-exceptional set and $F^{\prime}=(\pi^{\prime\prime})^{\ast}(F)_{red} \subset E$, where $E$ is the $f$-exceptional set. Then by~\cite{Art62},~\cite{Art66},  $F^{\prime}$ can be contracted over $X$ to a finite set of rational singularities. Then there is a factorization
\[
\xymatrix{
     &  X^{\prime}\ar[dr]^{f} & \\
X^{\prime\prime}\ar[ur]^{f^{\prime}} \ar[rr]^{f^{\prime\prime}} & & X
}
\]
where $f^{\prime}$ is the contraction of $F^{\prime}$. Then $D^{\prime\prime}$ descends to a vector field $D^{\prime}$ on $X^{\prime}$.  Let $\pi^{\prime} \colon X^{\prime} \rightarrow \bar{Y}$ be the quotient. Then it is easy to see that there is a commutative diagram
\begin{gather*}
\xymatrix{
 & X^{\prime\prime}\ar[r]^{f^{\prime}} \ar[d]_{\pi^{\prime\prime}} & X^{\prime} \ar[d]^{\bar{\pi}} \\
 & Y^{\prime\prime} \ar[r]^{\bar{g}} & \bar{Y} \\
}
\end{gather*}
From the commutativity of the above diagram it follows that $\bar{g}$ contracts $F$. Therefore, $\bar{Y}$ is the minimal resolution $Y^{\prime}$ of $Y^{\prime\prime}$. Moreover, since the maps $f^{\prime}$ and $g^{\prime}$ are over $X$ and $Y$ respectively, the previous diagram gives the following commutative diagram
\begin{gather}\label{sec5-diagram-2}
\xymatrix{
 & X^{\prime}\ar[r]^f \ar[d]^{\pi^{\prime}} & X \ar[d]^{\pi} \\
Z & Y^{\prime} \ar[l]_{\phi}\ar[r]^g & Y \\
}
\end{gather}
where now $Y^{\prime}$ is the minimal resolution of $Y$ and $Z$ its minimal model. However, $X^{\prime}$ may now be singular. In any case though $X^{\prime}$ has rational singularities and the $f$ and $g$ exceptional sets are trees of smooth rational curves. Suppose that
\begin{gather}\label{sec5-eq-1}
K_{Y^{\prime}}=g^{\ast} K_Y -F_1\\
K_{Y^{\prime}}=\phi^{\ast}K_Z+F_2 \nonumber
\end{gather}
where $F_1$, $F_2$ are effective $g$ and $\phi$-exceptional divisors, reespectively. Let $\Delta$ be the divisorial part of $D$. Unlike the multiplicative case, $\Delta$ may be singular, nonreduced and it may even contain isolated fixed points of $D$. By adjunction for purely inseparable morphisms~\cite{Ek87}~\cite{R-S76},
\begin{gather}\label{sec5-eq-2}
K_X=\pi^{\ast}K_Y+\Delta.
\end{gather}
Moreover, from Proposition~\ref{structure-of-pi}, $\pi$ is a torsor over a codimension 2 open subset of $Y$. Moreover, since $Y^{\prime}$ is smooth and $X^{\prime}$ is normal, $\pi^{\prime}$ is a torsor too. In particular $X^{\prime}$ has hypersurface singularities and hence $K_{X^{\prime}}$ is Cartier. Moreover,  by Proposition~\ref{structure-of-pi}, there are exact sequence
\begin{gather}\label{sec5-eq-3}
0 \rightarrow \mathcal{O}_Y \rightarrow E \rightarrow L^{-1} \rightarrow 0 \\
0 \rightarrow \mathcal{O}_{Y^{\prime}} \rightarrow E^{\prime} \rightarrow M^{-1} \rightarrow 0 \nonumber
\end{gather}
where $E=\pi_{\ast}\mathcal{O}_X$, $E^{\prime}=\pi^{\prime}_{\ast}\mathcal{O}_{X^{\prime}}$, $L=\mathcal{O}(C)$ is a reflexive sheaf on $Y$ and $M=\mathcal{O}_{Y^{\prime}}(C^{\prime})$ is an invertible sheaf on $Y^{\prime}$. Moreover,
\begin{gather}\label{sec5-eq-4}
K_X=\pi^{\ast}(K_Y+C)\\
K_{X^{\prime}}=(\pi^{\prime})^{\ast}(K_{Y^{\prime}}+C^{\prime})
\end{gather}
From this and~\ref{sec5-eq-2} it follows that $\pi^{\ast}C=\Delta$. Moreover, since $K_X$ is ample and $\pi$ a finite morphism, it follows that $K_Y+C$ is ample too.

Finally, from~\ref{sec5-eq-4} and the fact that $\pi$ is finite of degree 2 it follows that
\begin{gather}\label{sec4-eq-5}
K_X^2=2(K_Y+C)^2=2K_Y\cdot (K_Y+C) +2C\cdot (K_Y+C).
\end{gather}

As in the multiplicative case, the proof of the Theorem~\ref{additive-type} will be in several steps, according to the Kodaira dimension $k(Y)$ of $Y$.

\textbf{Case 1.} Suppose $k(Y)=2$ or $k(Y)=1$. In this case I will show that $K_Y^2 \geq 3$. Suppose that $\Delta \not=0$. Then arguing similarly as in Cases 1 and 2 in the proof of Theorem~\ref{mult-type}, we get that $K_X^2 \geq 3$. 

Suppose that $\Delta=0$. Then $K_X=\pi^{\ast}K_Y$ and hence $K_Y$ is ample. In the case of vector fields of multiplicative type, if this happenned then $k(Y)=2$. This happenned because $Y$ had singularities of type $A_1$, in particular rational. If we knew that $Y$ had rational singularities it would be possible to show that $k(Y)=2$ again by comparing $H^0(\omega_Y^{[n]})$ and $H^0(\omega_{Y^{\prime}}^n)$. However, $Y$ may have non rational singularities in the additive case so in principle it could happen that $\Delta=0$ and $k(Y)< k(X)$.

Next I will show that in any case, $K_X^2\geq 4$. By its construction, $\pi$ factors through the geometric Frobenious $F \colon X \rightarrow X^{(2)}$. In fact there is a commutative diagram
\[
\xymatrix{
     & Y \ar[dr]^{\nu} \\
X \ar[ur]^{\pi}\ar[rr]^F & & X^{(2)}
}
\]
Since $X^{(2)}$ is smooth and $Y$ is normal, then $\nu$ is a torsor over $X^{(2)}$~\cite{Ek87}. Therefore, \[
K_Y=\nu^{\ast}(K_{X^{(2)}}+W^{(2)})
\]
where $W^{(2)}$ is a divisor on $X^{(2)}$. Recall that the geometric Frobenious is constructed from the next commutative diagram
\[
\xymatrix{
X  \ar[drr]^{F_{ab}}\ar[ddr] \ar[dr]^F     &                                 &        \\     
         &   X^{(2)} \ar[r]^{pr_1}\ar[d]^{pr_2}                      &       X\ar[d]^{\pi}   \\
         &   \mathrm{Spec}(k) \ar[r]^{F_{ab}}             &   \mathrm{Spec}(k)
}
\]
where $F_{ab}$ is the absolute Frobenious. Since $k$ is algebraically closed, $pr_1$ is an isomorphism. Hence $W^{(2)}=pr_1^{\ast}W$, where $W$ is a divisor on $X$. Then 
\[
K_X=\pi^{\ast}K_Y=\pi^{\ast}\nu^{\ast}(K_{X^{(2)}}+W^{(2)})=F^{\ast}(K_{X^{(2)}}+W^{(2)})=F^{\ast}_{ab}(K_X+W)=2K_X+2W.
\]
Therefore $K_X=-2W$ and hence $K_X^2 =4W^2 \geq 4$, as claimed.

\textbf{Case 2.} Suppose that $k(Y)=0$. In this I will show that one of the following happens
\begin{enumerate}
\item $K_X^2 \geq 3$.
\item $1\leq K_X^2\leq 2$, $X$ is unirational and $\pi_1^{et}(X)=\{1\}$. In particular, if $K_X^2=1$, then $X$ is a simply connected supersingular Godeaux surface. 
\end{enumerate}

I will only prove the statement for the case when $K_X^2=1$. The proof for case $K_X^2=2$ is similar with minor differences and it is left to the reader.

Suppose then that $K_X^2=1$. Similar arguments as in the cases $k(Y)=2$ and $k(Y)=1$ show that if $\Delta=0$ then $K_X^2 \geq 4$.

Suppose that $Y^{\prime}$, in diagram~\ref{sec5-diagram-2} is not minimal. Then similar arguments as in the multiplicative case give that $K_X^2 \geq 3$. 

Suppose now that $Y^{\prime}$ is minimal. The argument of the multiplicative case used in an essential way the fact that $Y$ has singularities of type $A_1$ and cannot be used in this case directly.  

Suppose that $K_X^2=1$. Then it is known~\cite{Li09} that $1\leq \chi(\mathcal{O}_X) \leq 3$ ($0\leq \chi(\mathcal{O}_X) \leq 4$, if $K_X^2=2$~\cite[Corollary1.8]{Ek87}). Hence in order to show the claim it suffices to show that the cases $\chi(\mathcal{O}_X)\in\{2,3\}$ is impossible, that $\pi_1^{et}(X)=\{1\}$ and that $X$ is unirational and not singular. FromLemma~\ref{b1} it follows that $b_1(X)=0$ and hence $b_1(Y^{\prime})=b_1(Y)=b_1(X)=0$. Hence $Y^{\prime}$ is either an Enriques or a K3 surface and hence $c_2(Y^{\prime})=12$, if $Y^{\prime}$ is Enriques, and 24 if it is K3.

From diagram~\ref{sec5-diagram-2} it follows that 
\begin{gather}\label{sec5-eq-6}
c_2(Y^{\prime})=\chi_{et}(Y^{\prime})=\chi_{et}(X^{\prime})=c_2(X)+k,
\end{gather}
where $k$ is the number of $f$-exceptional curves.

Suppose that $\chi(\mathcal{O}_X)=3$. Then from Noethers formula we get that $c_2(X)=35$ and from~\ref{sec5-eq-6} that $c_2(Y^{\prime})=35+k>24$. Hence this case is impossible.

Suppose that $\chi(\mathcal{O}_X)=2$. Then from Noethers formula we get that $c_2(X)=23$. Then $c_2(Y^{\prime})=c_2(X)+k=23+k$. Hence the only possibility is that $Y^{\prime}$ is a K3 and $k=1$. Then I claim that $Y^{\prime}$ has exactly one singular point which must be canonical of type $A_1$. $Y$ has canonical singularities since $K_{Y^{\prime}}=0$. Let $E$ and $F$ be the $f$ and $g$-exceptional curves, respectively. Both are smooth rational curves. Then since $K_{Y^{\prime}}=0$, it follows that $F^2=-2$ and so $Y$ has exactly one $A_1$ singular point (if $K_X^2=2$ then $k=2$ and $Y$ has canonical singularities whose minimal resolution has two exceptional curves. Then by Proposition~\ref{prop4}, $Y$ has exactly two $A_1$ singular points).

Let $\Delta$ be the divisorial part of $D$ and $\Delta^{\prime}$ the divisorial part of $D^{\prime}$, the lifting of $D$ on $X^{\prime}$. Then since $Y^{\prime}$ is K3, $K_X=\Delta$ and $K_{X^{\prime}}=\Delta^{\prime}$. Moreover $\Delta^{\prime}=(\pi^{\prime})^{\ast}M$, where $M$ is as in equations~\ref{sec5-eq-3}. From~\ref{sec5-eq-3} it follows that
\[
\chi(M^{-1})=\chi(E^{\prime})-\chi(\mathcal{O}_{Y^{\prime}})=\chi(\mathcal{O}_{X^{\prime}})-\chi(\mathcal{O}_{Y^{\prime}})=2-2=0.
\]
Now Riemann-Roch gives that
\[
0=\chi(M^{-1})=\chi(\mathcal{O}_{Y^{\prime}})+\frac{1}{2}(M^2+M\cdot K_{Y^{\prime}})=2+\frac{1}{2}M^2.
\]
Hence $M^2=-4$ and therefore, since $\Delta^{\prime}=(\pi^{\prime})^{\ast}M$, $K_{X^{\prime}}^2=(\Delta^{\prime})^2=-8$. Since $X$ is smooth and $K_{X^{\prime}}$ Cartier, there is a positive $a\in \mathbb{Z}$ such that 
\begin{gather}\label{sec5-eq-7}
K_{X^{\prime}}=f^{\ast}K_X+aE.
\end{gather}
Now $E$ may or may not be an integral curve for $D^{\prime}$. If it is an integral curve, then $(\pi^{\prime})^{\ast}F=E$ and hence $E^2=-4$. Then from~\ref{sec5-eq-7} we get that 
\begin{gather}\label{sec5-eq-8}
-8=K_{X^{\prime}}^2=K_X^2-4a^2=1-4a^2,
\end{gather}
which is impossible. Suppose that $E$ is not an integral curve for $D^{\prime}$. Then $2E=(\pi^{\prime})^{\ast}F$ and hence $E^2=-1$. Then from~\ref{sec5-eq-8} it follows that $a=3$ and hence $K_{X^{\prime}}=f^{\ast}K_X+3E$. Hence $3E$ is Cartier. But since $2E=(\pi^{\prime})^{\ast}F$ and $Y^{\prime}$ is smooth, it follows that $2E$ is Cartier as well. Hence $E$ is Cartier. But since $E=\mathbb{P}^1$ it follows that $X^{\prime}$ is in fact smooth and $f$ is the contraction of a $-1$ curve. But then $K_{X^{\prime}}=f^{\ast}K_X+E$, a contradiction. Hence the case $\chi(\mathcal{O}_X)=2$ is impossible too. Hence $\chi(\mathcal{O}_X)=1$ and therefore $X$ is a Godeaux surface. 

Next I will show that $Y^{\prime}$ is a K3 surface with a special configuration~\cite{SB96} of rank 13. 

Again as before we find that
\begin{gather}\label{sec5-eq-9}
c_2(Y^{\prime})=c_2(X)+k=11+k,
\end{gather}
where $k$ is the number of $f$-exceptional curves. If $Y^{\prime}$ was Enriques, then $c_2(Y^{\prime})=12$ and hence $k=1$. I now repeat the previous argument. Exactly as before we get that $K_{X^{\prime}}^2=-4$. Suppose that $(\pi^{\prime})^{\ast}F=E$. Then $E^2=-4$ and hence from~\ref{sec5-eq-8} we get that $-4=1-4a^2$, which is impossible. If on the other hand $(\pi^{\prime})^{\ast}F=2E$, then $E^2=-1$ and hence again from~\ref{sec5-eq-8} we get that $-4=1-a^2$ and hence $a^2=5$, which is again impossible since $a\in \mathbb{Z}$. Hence $Y^{\prime}$ is a $K3$ surface and therefore $c_2(Y^{\prime})=24$. Then from~\ref{sec5-eq-9} it follows that $k=13$ and hence $g$ has exactly 13 exceptional curves. Moreover, $Y$ has canonical singularities and therefore by~\ref{prop4} they must be of type either $A_1$ or $D_{2n}$. Hence by~\cite{SB96}, $Y^{\prime}$ has a special configuration of rank 13. Then $Y^{\prime}$ is unirational~\cite{SB96} and hence $X$  is unirational as well. Moreover considering that the \'
etale fundamental group is a birational invariant between smooth varieties and also invariant under purely inseparable finite maps, we get from~\ref{sec5-diagram-1} that 
\[
\pi_1^{et}(X)=\pi_1^{et}(X^{\prime\prime})=\pi_1^{et}(Y^{\prime\prime})=\pi_1^{et}(Y^{\prime})=\{1\}.
\]
Finally, I will show that $X$ is supersingular. This means that $\mathrm{p}_g(X)=h^1(\mathcal{O}_X)=1$ and the induced map $F^{\ast}$ of the Frobenious on $H^1(\mathcal{O}_X)$ is zero. Since $\omega_Y\cong \mathcal{O}_Y$ we get from duality for finite morphisms~\cite{Ha77} that $\omega_X=\pi^{!}\mathcal{O}_Y$. Hence
\[
H^0(\omega_X)=\mathrm{Hom}_Y(\pi_{\ast}\omega_X,\mathcal{O}_Y).
\]
But this is nonzero since the map $\phi \colon \pi_{\ast}\omega_X \rightarrow \mathcal{O}_Y$ defined by $\phi(a)=Da$ is nonzero and $\mathcal{O}_Y$-linear. Then from~\cite[Corollary 1.8]{Ek87} it follows that $\mathrm{p}_g(X)=h^1(\mathcal{O}_X)=1$. Finally for exactly the same reasons as in the multiplicative case, $F^{\ast}$ is zero and hence $X$ is supersingular.

\textbf{Case 3.} Suppose that $k(Y)=-\infty$. In this case I will show the following.
\begin{enumerate}
\item If $K_X^2=2$, then $X$ is uniruled. Moreover, if $\chi(\mathcal{O}_X)\geq 2$, then $X$ is unirational and $\pi_1^{et}(X)=\{1\}$.
\item If $K_X^2=1$ then $X$ is unirational, simply connected and $p_g(X) \leq 1$. In particular 
$1 \leq \chi(\mathcal{O}_X) \leq 2$.
\end{enumerate}

This case is very different from the corresponding multiplicative case where the core of the argument was that if $K_X^2 \leq 2$ then $X$ lifts to characteristic zero and hence Kodaira-Nakano vanishing holds which implies that $X$ does not have any nonzero global vector fields. The proof that $X$ lifts to characteristic zero was based on showing that every space and map that appears in~\ref{sec5-diagram-2} lifts to charracteristic zero.  However, even though in this case too $Y^{\prime}$ lifts to characteristic zero, the construction of $X^{\prime}$ as a torsor over $Y^{\prime}$ depends heavily on being in positive characteristic and does not necessarily lift to characteristic zero.

The first part of the claim is obvious. Only the statement that if $\chi(\mathcal{O}_X)\geq 2$, then $X$ is unirational and $\pi_1^{et}(X)=\{1\}$ needs some justification. From Lemma~\ref{b1} it follows that if $\chi(\mathcal{O}_X)\geq 2$, then $b_1(X)=0$. Hence $b_1(X^{\prime})=0$ and therefore $X^{\prime}$ is rational and hence $X$ unirational and simply connected.

Suppose in the following that $K_X^2=1$. Fix notation as in diagram~\ref{sec5-diagram-2}. Since $k(Y)=-\infty$, then $Z$ is ruled over a smooth curve $B$. From Lemma~\ref{b1} it follows that $b_1(X)=0$ and hence $B\cong \mathbb{P}^1$. Hence $X$ is unirational and for the same reasons as in the previous cases, $\pi_1^{et}(X)=\pi_1^{et}(Z)=\{1\}$. Hence $X$ is simply connected. It remains to show that $p_g(X) \leq 1$ and that $1 \leq \chi(\mathcal{O}_X) \leq 2$. Since $X$ is canonically polarized, Noether's inequality gives that $K_X^2\geq 2\mathrm{p}_g(X)-4$. Hence if $K_X^2=1$ then $\mathrm{p}_g(X) \leq 2$. Moreover 
from~\cite[Corollary 1.8]{Ek87}, $1\leq \chi(\mathcal{O}_X) \leq 3$. Hence to show the claim it suffices to show that the case $\mathrm{p}_g(X)=2$ is impossible. If this happens, then by~\cite[Corollary 1.8]{Ek87}, then $\chi(\mathcal{O}_X)\in \{2,3\}$.

Suppose then that $\mathrm{p}_g(X)=2$. 

For the same reasons as in Case 1., if $\Delta=0$ then $K_X^2\geq 4$. So we may assume that $\Delta\not=0$.

Next I will show that $Y$ has exactly one singular point which is of type $A_1$. 

Let $C_1, C_2 \in |K_X|$ be two general members. Then since $K_X$ is ample and $K_X \cdot C_i=1$, it follows that $C_i$ are irreducible and reduced curves, $i=1,2$. Moreover, $C_1^2=C_2^2=C_1\cdot C_2=1$. Let $P=C_1 \cap C_2$.

\textbf{Claim.} 
\begin{enumerate}
\item $P$ is a smooth point of both $C_1$ and $C_2$.
\item There are smooth rational curves $\tilde{C}_i \in |K_Y+C|$, $i=1,2$ such that $\pi^{\ast}\tilde{C}_i=C_i$, $i=1,2$.
\item $Q=\pi(P)$ is the unique singular point of $Y$. Moreover, it is of type $A_1$.
\end{enumerate}
Note that $Y$ is necessarily singular since from the equation $K_X=\pi^{\ast}(K_Y+C)$ we get that $(K_Y+C)^2=1/2$.

Let $f \colon W \rightarrow X$ be the blow up of $X$ at $P$. Let $E$ be the $f$-exceptional curve. Then
\begin{gather*}
f^{\ast}C_1=C^{\prime}_1+m_1E\\
f^{\ast}C_2=C^{\prime}_2+m_2E
\end{gather*}
Hence
\[
1=f^{\ast}C_1\cdot f^{\ast}C_2=C_1^{\prime}\cdot f^{\ast}C_2+m_1E\cdot f^{\ast}C_2=C_1^{\prime}\cdot C_2^{\prime}+m_1m_2.
\]
Hence $C_1^{\prime}\cdot C_2^{\prime}=0$ and $m_1=m_2=1$. Therefore $C_1$ and $C_2$ are smooth at $P$. This shows the first part of the claim.

Next I will show that there are $\tilde{C}_i \in |K_Y+C|$ such that $\pi^{\ast}\tilde{C}_i=C_i$, $i=1,2$. In the notation of diagram~\ref{sec5-diagram-1}, $H^0(\omega_{X^{\prime\prime}})=H^0(\omega_X)$. Let then $C_i^{\prime\prime}\in|K_{X^{\prime\prime}}|$ be liftings of $C_i$ in $X^{\prime\prime}$, $i=1,2$. Then I will show that there are $\bar{C}_i\in |K_{Y^{\prime\prime}}+M|$ such that $(\pi^{\prime\prime})^{\ast}\bar{C}_i=C_i^{\prime\prime}$. Then $\tilde{C}_i=h_{\ast}\bar{C}_i \in |K_Y+C|$ and $\pi^{\ast}\tilde{C}_i=C_i$, $i=1,2$. This is because all the previous equations hold over a codimension 2 open subset of $Y$ (its smooth part) and so everywhere. Now from~\ref{sec5-eq-4},
\[
H^0(\omega_{X^{\prime\prime}})=H^0((\pi^{\prime\prime})^{\ast}(\omega_{Y^{\prime\prime}}\otimes M)=H^0(\omega_{Y^{\prime\prime}}\otimes M\otimes \pi^{\prime\prime}_{\ast}\mathcal{O}_{X^{\prime\prime}})
\]
Then from~\ref{sec5-eq-3} we get the exact sequence
\[
0 \rightarrow \omega_{Y^{\prime\prime}}\otimes M \rightarrow \pi^{\prime\prime}_{\ast}\mathcal{O}_{X^{\prime\prime}}\otimes \omega_{Y^{\prime\prime}}\otimes M \rightarrow  \omega_{Y^{\prime\prime}} \rightarrow 0.
\]
This gives an exact sequence in cohomology
\[
0 \rightarrow H^0(\omega_{Y^{\prime\prime}}\otimes M) \rightarrow H^0(\pi^{\prime\prime}_{\ast}\mathcal{O}_{X^{\prime\prime}}\otimes \omega_{Y^{\prime\prime}}\otimes M) \rightarrow  H^0(\omega_{Y^{\prime\prime}}) \rightarrow \cdots
\]
Now since $Y^{\prime\prime}$ is rational, it follows that $H^0(\omega_{Y^{\prime\prime}})=0$. Hence 
\[
H^0(\omega_{Y^{\prime\prime}}\otimes M) = H^0(\pi^{\prime\prime}_{\ast}\mathcal{O}_{X^{\prime\prime}}\otimes \omega_{Y^{\prime\prime}}\otimes M)
\]
and therefore there are $\bar{C}_i\in |K_{Y^{\prime\prime}}+M|$ such that $(\pi^{\prime\prime})^{\ast}\bar{C}_i=C^{\prime\prime}_i$, $i=1,2$. This concludes the proof of the second part of the claim.

I will next show that $\tilde{C}_i\cong \mathbb{P}^1$, $i=1,2$. Since $\tilde{C}_i \in |K_Y+C|$, there exists an exact sequence
\[
0 \rightarrow \mathcal{O}_Y(-K_Y-C) \rightarrow \mathcal{O}_Y \rightarrow \mathcal{O}_{\tilde{C}_i} \rightarrow 0.
\]
This gives an exact sequence in cohomology
\[
\cdots \rightarrow H^1(\mathcal{O}_Y(-K_Y-C)) \rightarrow H^1(\mathcal{O}_Y) \rightarrow H^1(\mathcal{O}_{\tilde{C}_i}) \rightarrow H^2(\mathcal{O}_Y(-K_Y-C)) \cdots 
\]
I will show that $H^1(\mathcal{O}_Y)=H^2(\mathcal{O}_Y(-K_Y-C))=0$. Therefore, $H^1(\mathcal{O}_{\tilde{C}_i})=0$ and hence $\tilde{C}_i\cong \mathbb{P}^1$. Since $Y^{\prime}$ is rational it follows that $H^i(\mathcal{O}_{Y^{\prime}})=0$, $i=1,2$. If $Y$ had rational singularities then the same would hold for $Y$. But $Y$ may have nonrational singularities. 
However, the Leray spectral sequence we get the exact sequence
\[
0 \rightarrow H^1(g_{\ast}\mathcal{O}_{Y^{\prime}}) \rightarrow H^1(\mathcal{O}_{Y^{\prime}}) \rightarrow H^0(R^1g_{\ast}\mathcal{O}_{Y^{\prime}}) \rightarrow H^2(\mathcal{O}_Y^{\prime})
\]
Since $g_{\ast}\mathcal{O}_{Y^{\prime}}=\mathcal{O}_Y$ and $ H^1(\mathcal{O}_{Y^{\prime}})=0$, it follows that $H^1(\mathcal{O}_Y)=0$ as well. Next I will show that $H^2(\mathcal{O}_Y(-K_Y-C))=0$. By Serre duality for Cohen-Macauley sheaves~\cite{KM98}, we get that
\[
H^2(\mathcal{O}_Y(-K_Y-C))=H^0(\mathcal{O}_Y(2K_Y+C)).
\]
Suppose that $H^0(\mathcal{O}_Y(2K_Y+C))\not= 0$. Then there exists a nonzero effective divisor $Z \in |2K_Y+C|$. Then $K_Y+C = Z-K_Y$. Therefore from~\ref{sec5-eq-4} we get that 
\begin{gather}
K_X^2=2(K_Y+C)^2=2K_Y\cdot(K_Y+C) +2C\cdot (K_Y+C) = \label{sec5-eq-10}\\
2Z\cdot (K_Y+C)-2K_Y\cdot (K_Y+C).\nonumber
\end{gather}
Since $K_Y+C$ is ample and $2Z, 2C$ are Cartier, $2Z\cdot (K_Y+C)\geq 1$ and $2C\cdot (K_Y+C) \geq 1$ (since $2C$ is equivalent to $\pi_{\ast}\Delta$, which is effective). 

Suppose that $K_Y \cdot (K_Y+C) <0$. Then from the second equality of~\ref{sec5-eq-10} it follows that $K_X^2 \geq 2$. Suppose that $K_Y \cdot (K_Y+C) >0$. Then from the first equality of~\ref{sec5-eq-10} it follows that $K_X^2\geq 3$. 

Suppose now that $K_Y \cdot (K_Y+C)=0$. Then $\pi^{\ast}K_Y \cdot K_X=0$ and hence from the adjunction $K_X=\pi^{\ast}K_Y+\Delta$ we get that $K_X \cdot \Delta=K_X^2=1$.

Fix notation as in diagram~\ref{sec5-diagram-1}. Then from Corollary~\ref{sec3-cor-1} we get that
\begin{gather*}
K_{X^{\prime\prime}}\cdot \Delta^{\prime\prime}=4\left(\chi(\mathcal{O}_{X^{\prime\prime}})-2\chi(\mathcal{O}_{Y^{\prime\prime}})\right).
\end{gather*}
Since $Y^{\prime\prime}$ is rational, $\chi(\mathcal{O}_{Y^{\prime\prime}})=1$ and hence
\begin{gather}\label{sec5-eq-11}
K_{X^{\prime\prime}}\cdot \Delta^{\prime\prime}=4\left(\chi(\mathcal{O}_{X^{\prime\prime}})-2\right)
\end{gather}
If $\mathrm{p}_g(X)=2$, then either $\chi(\mathcal{O}_X)=3$ or $\chi(\mathcal{O}_X)=2$. Suppose that $\chi(\mathcal{O}_X)=3$. Then from~\ref{sec5-eq-11} it follows that $K_{X^{\prime\prime}}\cdot \Delta^{\prime\prime}=4>1=K_X\cdot \Delta$, which is impossible by Corollary~\ref{sec3-cor-1}.

Suppose that $\chi(\mathcal{O}_X)=2$. In this case, $K_{X^{\prime\prime}}\cdot \Delta^{\prime\prime}=0$.

\begin{claim}\label{sec5-claim} In this case, $Y$ has exactly one singular point which must be of type $A_1$. 
\end{claim}

From Corollary~\ref{sec3-cor-1}, $K_X \cdot \Delta$ decreases from $X$ to $X^{\prime\prime}$. In order to show the claim I will study how exactly $K_X \cdot \Delta$ decreases. Since $f$ is a composition of blow ups of isolated singular points of $D$, it suffices to examine what happens after a single blow up. So let $f_1 \colon X_1 \rightarrow X$ be the blow up of an isolated singular point of $D$. Let $D_1$ be the lifting of $D$ on $X_1$,  $\Delta_1$ its divisorial part and $Y_1$ the quotient of $X_1$ by $D_1$. Then there exists a commutative diagram
\begin{gather}\label{sec5-diagram-4}
\xymatrix{
 & X_1\ar[r]^{f_1} \ar[d]_{\pi_1} & X \ar[d]^{\pi} \\
 & Y_1 \ar[r]^{g_1} & Y \\
}
\end{gather}
Let $E$ be the $f$-exceptional curve and $F=\pi_1(E)$ the $g$-exceptional curve. Then, since $K_Y$ is Cartier, there is $a \in \mathbb{Z}$ such that
\[
K_{Y_1}=g_1^{\ast}K_Y+aF.
\]
Moreover, $K_X=\pi^{\ast}K_Y+\Delta$. Hence from~\ref{sec5-diagram-4} we that
\[
K_{X_1}=f_1^{\ast}K_X+E=f_1^{\ast}\pi^{\ast}K_Y+f_1^{\ast}\Delta+E=\pi_1^{\ast}K_{Y_1}+f_1^{\ast}\Delta+(1-ka)E
\]
where $k\in\{1,2\}$ is such that $\pi_1^{\ast}F=kE$. If $E$ is an integral curve for $D^{\prime}$, then $k=1$. Otherwise $k=2$. Therefore,
\begin{gather}\label{sec5-eq-16}
\Delta_1=f_1^{\ast}\Delta+(1-ka)E.
\end{gather}
From the proof of Proposition~\ref{K-decreases} it follows that $1-ka\geq 0$. If $1-ka>1$, then 
\[
K_{X_1}\cdot \Delta_1=K_X\cdot \Delta -(1-ka)\leq  K_X\cdot \Delta -2.
\]
But considering that $K_X\cdot \Delta =1$ and $K_{X^{\prime\prime}}\cdot \Delta^{\prime\prime}=0$, this cannot happen. Hence $1-ka\in \{0,1\}$.

Suppose that $1-ka=1$. Then $a=0$ and $K_{X_1}\cdot \Delta_1=K_X\cdot \Delta -1$, hence $K_X\cdot \Delta$ drops by one. Again considering that $K_X\cdot \Delta =1$ and $K_{X^{\prime\prime}}\cdot \Delta^{\prime\prime}=0$, this case can happen only once. In all other cases, $K_X\cdot \Delta$ stays the same and hence $1-ka=0$, i.e, $a=1/k$ (and hence $k=1$ since $a\in \mathbb{Z}$). 

In the notation then of~\ref{sec5-diagram-1}, one can write
\[
K_{Y^{\prime\prime}}=h^{\ast}K_Y+\sum_ia_iF_i,
\]
such that $a_i \in \mathbb{Z}$ and there is exactly one $i$ such that $a_i=0$ and $a_j >0$ for all $j\not=i$. Hence $Y$ has canonical singularities. Now consider the commutative diagram
\[
\xymatrix{
                & Y^{\prime} \ar[dr]^{g}      &\\
Y^{\prime\prime} \ar[ur]^{\phi} \ar[rr]^{h} & & Y \\
}
\]  
where $g \colon Y^{\prime} \rightarrow Y$ is the minimal resolution. Then since $h$ has exactly one crepant divisor, $g$ has exactly one crepant divisor too. Therefore since $Y$ has canonical singularities, it follows that $K_Z=g^{\ast}K_Y$ and $g$ contracts exactly one smooth rational curve of self intersection -2. Hence $Y$ has exactly 1 singular point which must be of type $A_1$, as claimed.

Then \[
c_2(Y^{\prime})=\chi_{et}(Y)+1=c_2(X)+1.
\]
If $\chi(\mathcal{O}_X)=2$, then $c_2(X)=23$. Therefore, $c_2(Y^{\prime})=24$. However, since $Y^{\prime}$ is rational, $\chi(\mathcal{O}_{Y^{\prime}})=1$ and hence from Noethers formula
\[
K_Y^2=K_{Y^{\prime}}^2=12-24=-12.
\]
Moreover, since $K_X=\pi^{\ast}K_Y+\Delta$ it follows that
\[
-24=2K_Y^2=(K_X-\Delta)^2=K_X^2+\Delta^2-2K_X\cdot \Delta =-1+\Delta^2
\]
and hence $\Delta^2=-23$. However, since $K_X \cdot \Delta =1$ and $K_X$ is ample, it follows that $\Delta $ is irreducible and reduced. But then from the genus formula
\begin{gather}\label{sec5-eq-222}
2p_a(\Delta)-2=K_X\cdot \Delta +\Delta^2=1-23=-22.
\end{gather}
But this is impossible. Therefore $K_Y\cdot (K_Y+C) \not=0$ and hence
\[
H^0(\mathcal{O}_Y(2K_Y+C))\not= 0,
\]
as claimed. Hence $\tilde{C}_i\cong \mathbb{P}^1$, $i=1,2$. In particular, they are both smooth.

Next I will show that $Q=\pi(P)$, where $P=C_1\cap C_2$ is the only singular point of $Y$. Indeed. $P$ is the only base point of the 2-dimensional linear system $|K_X|$. Hence $Q$ is the only base point of $|K_Y+C|$. Since $\tilde{C}_1+\tilde{C}_2 \in |2K_Y+2C|$, it is Cartier. Moreover, since both $\tilde{C}_1$ and $\tilde{C}_2$ are smooth, it follows that $Y$ is smooth everywhere except $Q$ (as explained earlier $Y$ is singular and therefore cannot be smooth at $Q$). Hence $Y$ has exactly one singular point.

Next I will show that $Q\in Y$ is an $A_1$ point. Let $f_1 \colon X_1 \rightarrow X$ be the blow up of $P$. Let $D_1$ be the lifting of $D$ on $X_1$,  $\Delta_1$ its divisorial part and $Y_1$ the quotient of $X_1$ by $D_1$. Then there exists a commutative diagram
\begin{gather}\label{sec5-diagram-5}
\xymatrix{
 & X_1\ar[r]^{f_1} \ar[d]_{\pi_1} & X \ar[d]^{\pi} \\
 & Y_1 \ar[r]^{g_1} & Y \\
}
\end{gather}
Let $E$ be the $f$-exceptional curve and $F=\pi_1(E)$ the $g$-exceptional curve. Then there is $a \in \mathbb{Z}$ such that $K_{Y_1}=g_1^{\ast}K_Y+aF$. I will show that $Y_1$ is smooth, $a=0$ and $F^2=-2$.

Let $\bar{C}_i$ be the birational transforms of $C_i$ in $Y_1$, $i=1,2$. Then I claim that
\begin{gather}\label{sec5-eq-14}
g^{\ast}\tilde{C}_i=\bar{C}_i+\frac{1}{2}F,
\end{gather}
for $i=1,2$. Indeed. Suppose that $g^{\ast}\tilde{C}_i=\bar{C}_i+m_iF$. Then clearly $\pi_1^{\ast}\bar{C}_i=C_i^{\prime}$. Hence,
\[
C_i^{\prime}+E=f^{\ast}_1C_i=f_1^{\ast}\pi^{\ast}\tilde{C}_i=\pi_1^{\ast}g^{\ast}\tilde{C}_i=C_i^{\prime}+m_i\pi_1^{\ast}F.
\]
Therefore, $m_i\pi_1^{\ast}F=E$. This implies that $m_1=m_2$. Moreover, if $\pi_1^{\ast}F=E$, then $m_1=m_2=1$. If on the other hand, $\pi_1^{\ast}F=2E$, then $m_1=m_2=1/2$. Suppose that $\pi^{\ast}F=E$. Then $F^2=-1/2$. Also, 
\begin{gather}\label{sec5-eq-15}
g^{\ast}\tilde{C}_i=\bar{C}_i+F.
\end{gather}
Since $\tilde{C}_1\sim \tilde{C}_2$, it follows that $\bar{C}_1+F\sim \bar{C}_2+F$ and hence $\bar{C}_1 \sim \bar{C}_2$. Therefore, $\mathcal{O}_{Y_1}(\bar{C}_1)\cong \mathcal{O}_{Y_1}(\bar{C}_1)$. Considering now that $\bar{C}_1\cap \bar{C}_2=\emptyset$ it follows that both $\mathcal{O}_{Y_1}(\bar{C}_i)$, $i=1,2$, are invertible and hence $\bar{C}_1$ and $\bar{C}_2$ are Cartier. In addition. $\bar{C}_i^2=0$, $i=1,2$. But then from~\ref{sec5-eq-15} it follows that
\[
\bar{C}_i \cdot F=-F^2=1/2
\]
which is impossible since $\bar{C}_i$ are Cartier. Hence $\pi_1^{\ast}F=2E$ and $m_1=m_2=1/2$ and~\ref{sec5-eq-14} holds.

Next I will show that $Y_1$ is smooth. From~\ref{sec5-eq-14} it follows that
\[
g^{\ast}(\tilde{C}_1+\tilde{C}_2)=\bar{C}_1+\bar{C}_2+F.
\]
Since $\tilde{C}_1+\tilde{C}_2\in|2K_Y+2C|$, $\tilde{C}_1+\tilde{C}_2$ is Cartier and hence $\bar{C}_1+\bar{C}_2+F$ is Cartier as well. Since $\bar{C}_1$, $\bar{C}_2$ and $F$ are smooth, then the only possible singularities of $Y_1$ are at $\bar{C}_1\cap F$ and $\bar{C}_2\cap F$. I will show however that $\bar{C}_i$, $i=1,2$, are both Cartier and hence $Y_1$ is smooth as claimed. Since $2\tilde{C}_1 \sim \tilde{C}_1+\tilde{C}_2$, and are both Cartier, it follows that $g^{\ast}(2\tilde{C}_1)\sim g^{\ast}(\tilde{C}_1+\tilde{C}_2)$. Hence
\[
2\bar{C}_1+F \sim \bar{C}_1+\bar{C}_2+F
\]
and therefore $\bar{C}_1 \sim \bar{C}_2$. Hence $\mathcal{O}_{Y_1}(\bar{C}_1)\cong \mathcal{O}_{Y_1}(\bar{C}_2)$. But $\bar{C}_1 \cap \bar{C}_2 =\emptyset$ (because $C^{\prime}_1 \cap C^{\prime}_2 =\emptyset$, where $C_1^{\prime}, C_2^{\prime}$ are the birational transforms of $C_1, C_2$ in $X_1$). Hence $\mathcal{O}_{Y_1}(\bar{C}_i)$ are invertible, $i=1,2$, and therefore $\bar{C}_1$ and $\bar{C}_2$ are both Cartier as claimed and hence $Y_1$ is smooth. Therefore $Y$ has exactly one singular point which is of type $A_1$. Moreover the diagram~\ref{sec5-diagram-5} is its resolution.

Next I claim that 
\begin{gather}\label{sec5-eq-17}
K_X\cdot \Delta +\Delta^2 +c_2(X)=1.
\end{gather}
Indeed. With notation as in~\ref{sec5-diagram-5}, since $Y_1$ is smooth, $D_1$ has no isolated fixed points. Hence if $\Delta_1$ is its divisorial part, then from Proposition~\ref{size-of-sing}
\[
K_{X_1}\cdot \Delta_1 +\Delta_1^2+c_2(X_1)=0.
\]
But from~\ref{sec5-eq-16} we get that $\Delta_1=f_1^{\ast}\Delta +E$. Now from this, the fact that $c_2(X_1)=c_2(X)+1$, the previous equation and some straightforward calculations we get~\ref{sec5-eq-17}.

If $\mathrm{p}_g(X)=2$ then $\chi(\mathcal{O}_X)\in\{2,3\}$. 

Suppose that $\chi(\mathcal{O}_X)=2$. Then since $K_X^2=1$, from Noethers formula we get that $c_2(X)=23$. Hence $c_2(Y_1)=\chi_{et}(Y)+1=c_2(X)+1=24$. Then since $Y_1$ is rational, from Noethers formula for $Y_1$ it follows that 
\[
K_{Y_1}^2=12\chi(\mathcal{O}_{Y_1})-24=1-24=-12
\]
Hence since $f_1$ is crepant, $K_Y^2=K_{Y_1}^2=-12$. From the adjunction formula $K_Y=\pi^{\ast}K_Y+\Delta$ we get that
\[
-24=2K_Y^2=(K_X-\Delta)^2=1+\Delta^2-2K_X\cdot \Delta,
\]
and hence
\begin{gather}\label{sec5-eq-20}
\Delta^2-2K_X \cdot \Delta =-25.
\end{gather}
However, since $c_2(X)=23$,~\ref{sec5-eq-17} gives also that
\begin{gather}\label{sec5-eq-21}
K_X\cdot \Delta +\Delta^2=-22.
\end{gather}
Now from~\ref{sec5-eq-20} and~\ref{sec5-eq-21} it follows that $K_X \cdot \Delta =1$ and $\Delta^2=-23$. But now for the exactly the same reasons as in~\ref{sec5-eq-222}, this is impossible. Hence the case $\chi(\mathcal{O}_X)=2$ is impossible.

Suppose that $\chi(\mathcal{O}_X)=3$. Arguing similarly as before we get that
\begin{gather}\label{sec5-eq-22}
K_X\cdot \Delta =5\\
\Delta^2=-39. \nonumber
\end{gather}
I will show that these relations are impossible by examining all possible cases for the structure of $\Delta$ as a cycle. Suppose that 
\[
\Delta=\sum_{i=1}^kn_i \Delta_i,
\]
where $\Delta_i$ are distinct prime divisors. Since $K_X$ is ample and $K_X \cdot \Delta =5$, it follows that $k\leq 5$ and there are the following possibilities.
\begin{enumerate}
\item $\Delta$ is reduced. Then $\Delta $ has at most 5 irreducible components.
\item $\Delta$ is not reduced. Then there are the following possibilities
\begin{enumerate}
\item $\Delta =2\Delta_1+\Delta_2+\Delta_3+\Delta_4$, and $K_X\cdot \Delta_1=1$, for all $i$.
\item $\Delta =2\Delta_1+\Delta_2+\Delta_3$, and $K_X\cdot \Delta_1=K_X\cdot \Delta_2=1$, $K_X\cdot \Delta_3=2$.
\item $\Delta =2\Delta_1+\Delta_2$, and $K_X\cdot \Delta _1 =1$, $K_X\cdot \Delta_2=3$, or $K_X\cdot \Delta _1 =2$, $K_X\cdot \Delta_2=1$.
\item $\Delta=2\Delta_1+2\Delta_2+\Delta_3$, and $K_X \cdot \Delta_i=1$, for all $i$.
\item $\Delta=4\Delta_1+\Delta_2$, and $K_X\cdot \Delta_1=K_X\cdot \Delta_2=1$.
\item $\Delta = 5\Delta_1$, $K_X \cdot \Delta_1=1$.
\item $\Delta=3\Delta_1 +\Delta_2+\Delta_3$, and $K_X\cdot \Delta_i=1$, for all $i$.
\item $\Delta=3\Delta_1+2\Delta_2$, and $K_X\cdot \Delta_i=1$, for all $i$. 
\end{enumerate}
\end{enumerate}
The proof that the equations~\ref{sec5-eq-22} are impossible will be by studying each one of the above cases separately. In all cases except 2.h, there are no solutions to~\ref{sec5-eq-22} under the restriction coming from the genus formula $2\mathrm{p}_a(\Delta_i)-2=K_X\cdot \Delta_i+\Delta_i^2$. Next I will work the cases 1. and 2.h. The cases 2.a to 2.g are treated in exactly the same way as 1. but 2.h needs some deeper geometric argument since there is a numerical solution to~\ref{sec5-eq-22}.

Suppose then that $\Delta=\sum_{i=1}^k\Delta_i$, $i\leq 5$ is reduced. Then
\begin{gather*}
K_X\cdot \Delta +\Delta^2=\sum_{i=1}^k(K_X\cdot \Delta_i+\Delta_i^2)+2\sum_{1\leq i<j\leq k}\Delta_i \cdot \Delta_j \geq \\
\sum_{i=1}^k(K_X\cdot \Delta_i+\Delta_i^2)=\sum_{i=1}^k(2\mathrm{p}_a(\Delta_i)-2) \geq -2k \geq -10,
\end{gather*}
which is impossible since from~\ref{sec5-eq-22}, $K_X\cdot \Delta +\Delta^2=-34$.

Suppose now that $\Delta=3\Delta_1+2\Delta_2$, $K_X\cdot \Delta_1 =K_X\cdot \Delta_2=1$. Then from the genus formula it follows that $K_X\cdot \Delta_i +\Delta_i^2 \geq -2$ and hence $\Delta_i^2\geq -3$. Then it is easy to see that the only solutions to~\ref{sec5-eq-22} are $\Delta^2_i=-3$, $i=1,2$, and $\Delta_1\cdot \Delta_2 =0$. In particular $\mathrm{p}_a(\Delta_i)=0$ and hence $\Delta_i=\mathbb{P}^1$, $i=1,2$. 

Fix notation as in~\ref{sec5-diagram-2} and let $\tilde{\Delta_i}$, $i=1,2$, be the images of $\Delta_i$ in $Y$, with reduced structure. From the previous discussion, $X^{\prime}$ is the blow up of the unique isolated fixed point $P$ of $D$ and $Y^{\prime}$ the minimal resolution of the singularity $Q=\pi(P)\in Y$. Moreover, $Q\in Y$ is an $A_1$ singular point and in particular $K_Y$ is Cartier. Then
\[
K_X=\pi^{\ast}K_Y+3\Delta_1+2\Delta_2.
\]
From this it follows that 
\begin{gather}\label{sec5-eq-23}
\pi^{\ast}K_Y\cdot \Delta_1=10\\
\pi^{\ast}K_Y\cdot \Delta_2 =7\nonumber
\end{gather}
From the projection formula we get that
\begin{gather*}
K_Y \cdot \pi_{\ast}\Delta_1=10\\
K_Y\cdot \pi_{\ast}\Delta_2 =7.
\end{gather*}
Moreover, $\pi_{\ast}\Delta_i$ is equal to either $\tilde{\Delta}_i$ or $2\tilde{\Delta}_i$, depending on whether $\Delta_i$ is an integral curve for $D$ or not. From the second equation and since $K_Y$ is Cartier, it follows that $\pi_{\ast}\Delta_2=\tilde{\Delta}_2$ and hence $\Delta_2$ is not an integral curve for $D$. Hence $\pi^{\ast}\tilde{\Delta}_2=2\Delta_2$ and therefore, $\tilde{\Delta}_2^2=-6$ and $K_Y \cdot \tilde{\Delta}_2 = 7$. I will now show that $Q\in \tilde{\Delta}_2$. If $Q\not\in\tilde{\Delta}_2$, then $\tilde{\Delta}_2$ is in the smooth part of $Y$. Hence from the genus formula again it follows that 
\[
2\mathrm{p}_a(\tilde{\Delta}_2)-2=K_Y\cdot \tilde{\Delta}_2+\tilde{\Delta}_2^2=7-6=1,
\]
which is impossible. Hence $Q \in \tilde{\Delta}_2$. Let $\Delta^{\prime}_2$ be the birational transform of $\tilde{\Delta}_2$ in $Y^{\prime}$. Then
\[
g^{\ast}\tilde{\Delta}_2=\Delta_2^{\prime} +mF
\]
for some $m \in \mathbb{Z}$. Since $F^2=-2$ and $\tilde{\Delta}_2$ is smooth, it follows that $m=1/2$. But then
\[
(\Delta^{\prime}_2)^2=\tilde{\Delta}_2^2-\frac{1}{2}=-6-\frac{1}{2}\not\in \mathbb{Z},
\]
which is impossible since $Y^{\prime}$ is smooth. Therefore the case 2.h is impossible too.

\end{proof}

\section{Examples.}\label{examples}
As mentioned in the introduction, there are several examples by now of canonically polarized surfaces with or without non-trivial global vector fields. The most common method to obtain them is as quotients of a rational surface with a rational vector field. By this method one always obtains uniruled surfaces, but nonuniruled canonically polarized surfaces with vector also exist~\cite{SB96}. However, In characteristics $p\not= 2$, it is not known if non uniruled examples exist.

Smooth hypersurfaces in $\mathbb{P}^3_k$ of degree $\geq 5$ have no vector fields~\cite{MO67}. The proof given in~\cite{MO67} shows that $H^0(T_X)=0$ by using standard exact sequences of $\mathbb{P}^3_k$. However, like Case 4. of the proof of Theorem~\ref{mult-type} it also follows from the Kodaira-Nakano vanishing which holds in this case since any smooth hypersurface lifts to $W_2(k)$.

A Godaux surface $X$ with $\pi_1^{et}(X)=\{1\}$ has no vector fields. This follows from Theorems~\ref{mult-type},~\ref{additive-type}. In particular Godeaux surfaces $\pi_1^{et}(X)=\mathbb{Z}/5\mathbb{Z}$ do not have vector fields. It is known~\cite{La81} that a general Godeaux surface with $\pi_1^{et}(X)=\mathbb{Z}/5\mathbb{Z}$ in any characteristic $p\not= 5$ is the quotient of a smooth quintic in $\mathbb{P}^3$ by a free action of $\mathbb{Z}/5\mathbb{Z}$. Then the fact that $X$ has no global vector fields follows also from the fact that smooth hypersurfaces have no global vector fields. However, this proof only works for general $X$ while Theorems~\ref{mult-type},~\ref{additive-type} give the result for any $X$.

Smooth examples of canonically polarized surfaces were given W. Lang~\cite{La83} and N.I. Shepherd-Barron~\cite{SB96}. In particular for any Del Pezzo surface $Y$ and $H \in |-K_Y|$, N.I. Shepherd-Barron~\cite[Theorem 5.2]{SB96} constructed a surface $X$ which admits maps
\[
X \stackrel{\alpha}{\rightarrow} Z \stackrel{\beta}{\rightarrow} Y,
\]
such that $\alpha$, $\beta$ are purely inseparable of degree 2,  $K_X^2=s$, $Z$ is a K3 surface with $12+s$ nodes and $\alpha^{\ast}\beta^{\ast}(\frac{1}{2}H)\sim K_X$. If $Y=\mathbb{P}^1 \times \mathbb{P}^1$, then $K_Y=-2G$ and hence $K_X=\alpha^{\ast}\beta^{\ast}G$. Since $K_Z=0$, the adjunction for purely inseparable morphisms~\cite{Ek87} shows that $\alpha $ is a foliation over $Z$ defined by the subsheaf 
$L=\mathcal{O}_X(\alpha^{\ast}\beta^{\ast}G)$ of $T_X$. Since $L=\mathcal{O}_X(\alpha^{\ast}\beta^{\ast}G)$ has sections, $X$ has nontrivial global vector fields. Moreover, it is unirational and $K_X^2=8$. In fact I do not know of any examples of smooth canonically polarized surfaces with vector fields and $K^2<8$. 

Finally, Liedtke~\cite{Li08} has constructed a series of examples of uniruled surfaces of general type with arbitrary high $K^2$. These examples are quotients of $\mathbb{P}^1 \times \mathbb{P}^1$ by rational vector fields. However the resulting surfaces are only of general type and not canonically polarized. However, taking their canonical models one obtains examples of canonically polarized surfaces with canonical singularities and $K^2$ arbitrary high. 

This is what happens in my knowledge for smooth canonically polarized surfaces. However, singular surfaces should be studied too. Especially because they are important in the moduli problem of canonically polarized surfaces and in particular its compactification. 

If $X$ is allowed to be singular and there are no restrictions on the singularities, then $K_X^2$ can take any value. Next I will present two examples of singular surfaces with vector fields and low $K^2$. The first example shows that unlike the smooth case (Theorem~\ref{lifts-to-zero}), the property \textit{Lifts to characteristic zero} does not imply smoothness of the automorphism scheme in the singular case. The second example shows that even in the presence of the mildest possible singularities (like canonical), the property \textit{Smooth automorphism scheme} is not deformation invariant and cannot be used to construct proper Deligne-Mumford moduli stacks in positive characteristic. 

\begin{example}\label{ex1}
In this example  I will construct a singular canonically polarized surface $X$ with $K_X^2=1$. Moreover this surface lifts to characteristic zero and nevertheless, unlike smooth surfaces that lift to characteristic zero, it has vector fields.

Let $Y$ be the weighted projective space $\mathbb{P}_k(2,1,1)$. It is singular with one singularity locally isomorphic to $xy+z^2=0$. Let 
\[
\pi \colon X=\mathrm{Spec} \left( \mathcal{O}_Y \oplus \mathcal{O}_Y(-5) \right) \rightarrow Y
\]
be the 2-cyclic cover defined by $\mathcal{O}_Y(-5)$ and a general section $s$ of $\mathcal{O}_Y(5)^{[2]}=\mathcal{O}_Y(10)$. Then $X$ is normal, $K_X$ is ample and $K_X^2=1$. Moreover, $X$ is liftable to characteristic zero and has global vector fields of multiplicative type.

From the theory of weighted projective spaces it follows that
\[
H^0(L^{[2]})=H^0(\mathcal{O}_Y(10))=k[x_0,x_1,x_2]_{(10)},
\]
the space of homogeneous polynomials of degree 10 with weights, $w(x_0)=2$, $w(x_i)=1$, $i=1,2,3$. Let $s=x_0^5+f(x_1,x_2) \in k[x_0,x_1,x_2]_{(10)}$. Then I claim that for general choice of $f(x_1,x_2)$, $X$ is normal. Indeed, locally over a smooth point of $Y$, 
\[
\mathcal{O}_X=\frac{\mathcal{O}_Y[t]}{(t^2-s)}.
\]
It is now straightforward to check that for general choice of $f(x_1,x_2)$, $X$ is smooth. Hence $X$ is normal. 

Next I will show that $K_X$ is ample and $K_X^2=1$. Indeed, by the construction of $X$ as a 2-cyclic cover over $Y$,
\[
\omega_X=\pi^{\ast}(\omega_Y\otimes \mathcal{O}_Y(5))^{\ast\ast}=\pi^{\ast}\mathcal{O}_Y(1)^{\ast\ast}.
\]
From this it follows that $\omega_X^{[2]}=\pi^{\ast}\mathcal{O}_Y(2)$ and hence $K_X$ is ample and has index 2. Moreover,
\[
K_X^2=\frac{1}{4}c^2_1(\omega_X^{[2]})=\frac{1}{4}2c_1^2(\omega_Y^{[2]}\otimes \mathcal{O}_Y(10))=\frac{1}{2}c_1^2(\mathcal{O}_Y(2))=\frac{1}{2}4c_1^2(\mathcal{O}_Y(1))=\frac{1}{2}\cdot 4 \cdot \frac{1}{2}=1.
\]
Finally since $Y$, $\mathcal{O}_Y(5)$ and $s$ lift to characteristic zero, the construction of $X$ as a 2-cyclic cover lifts also to characteristic zero. However, by its construction as a 2-cyclic cover, $X$ has nontrivial global vector fields of multiplicative type.
\end{example}

\begin{example}\label{ex2}

In this example I will construct a singular surface $X$ with canonical singularities of type $A_n$ such that $K_X^2=5$, $X$ has nonzero global vector fields and moreover there is a flat morphism $f \colon \mathcal{X} \rightarrow C$, where $C$ is a curve of finite type over $k$ such that $X=f^{-1}(s)$ for some $s\in C$ and whose general fiber has no vector fields. Therefore the property \textit{Smooth automorphism scheme} is not deformation invariant (even in the presence of the mildest possible singularities)
and cannot be used to construct proper moduli stacks in positive characteristic. 

Let $X \subset \mathbb{P}^3_k$ be the quintic surface given by
\[
xy(x^3+y^3+z^3)+zw^4=0.
\]
It is not difficult to check that its singularities are locally isomorphic to $xy+zf(x,y,z)=0$ and therefore they are canonical of type $A_n$. Moreover, the equation of $X$ 
is invariant under the graded derivation $D=w\frac{\partial}{\partial w}$ of $k[x,y,z,w]$. which therefore induces a nonzero global vector field on $X$. $X$ is smoothable by $X_t$ given by 
\[
(1-t)\left( xy(x^3+y^3+z^3)+zw^4\right) +t\left( x^5+y^5+z^5+w^5\right)=0,
\]
$t \in k$. For $t \not=0$, $X_t$ is a smooth quintic surface in $\mathbb{P}^3_k$ and hence it has no global vector fields and therefore $\mathrm{Aut}(X_t)$ is smooth for $t\not=0$.

\end{example}

\end{document}